\documentclass[11pt]{article}
\usepackage{amsmath,amssymb,amsthm,amscd}

\newtheorem{theorem}{Theorem}[section]
\newtheorem{lemma}[theorem]{Lemma}
\newtheorem{proposition}[theorem]{Proposition}
\newtheorem{corollary}[theorem]{Corollary}

\theoremstyle{plain}

\theoremstyle{definition}
\newtheorem{definition}[theorem]{Definition}

\numberwithin{equation}{section}

\renewcommand{\theenumi}{(\roman{enumi})}
\renewcommand{\labelenumi}{\textup{(\theenumi)}}

\title{K-theoretic duality for extensions of  Cuntz--Krieger algebras \\
}
\author{Kengo Matsumoto \\
Department of Mathematics \\
Joetsu University of Education \\
Joetsu, 943-8512, Japan
}

\begin{document}
%\date{2020, Nov 22}

\maketitle

\date{}

\def\det{{{\operatorname{det}}}}

%\maketitle
\begin{abstract}
We introduce the notion of K-theoretic duality 
for extensions of separable unital nuclear $C^*$-algebras
by using K-homology long exact sequence and 
cyclic six term exact sequence for K-theory groups of extensions.
We then prove that the Toeplitz extension  $\mathcal{T}_A$ 
of a Cuntz-Krieger algebra  $\mathcal{O}_A$ 
is the K-theoretic dual of the Toeplitz extension $\mathcal{T}_{A^t}$ 
of the Cuntz-Krieger algebra $\mathcal{O}_{A^t}$ for the
transposed matrix $A^t$ of $A$.
A pair of isomorphic Cuntz--Krieger algebras 
 $\mathcal{O}_A$ and  $\mathcal{O}_B$
does not necessarily yield the isomorphic pair
of $\mathcal{O}_{A^t}$ and  $\mathcal{O}_{B^t}.$
However, 
as an application, 
%two extensions $\mathcal{T}_A$ of $\mathcal{O}_A$ and $\mathcal{T}_B$ of $\mathcal{O}_B$
%are isomorphic if and only if their transposed extensions
%$\mathcal{T}_{A^t}$ of $\mathcal{O}_{A^t}$ and $\mathcal{T}_{B^t}$ of $\mathcal{O}_{B^t}$
%are isomorphic.More simply, 
we may show that 
two Toeplitz  algebras 
$\mathcal{T}_A$
and
$\mathcal{T}_B$
are isomorphic as $C^*$-algebras 
if and only if the Toeplitz algebras 
$\mathcal{T}_{A^t}$
and
$\mathcal{T}_{B^t}$ of their transposed matrices
are isomorphic.
\end{abstract}



%%%%%%%%%%%%%%%%%%%%%%%%%%%%%%
%2010{\it Mathematics Subject Classification}:
% Primary 46L55; Secondary 46L35, 37B10.

%{\it Keywords and phrases}:
%Topological Markov shifts, orbit equivalence,  Cuntz-Krieger algebras

\newcommand{\Ker}{\operatorname{Ker}}
\newcommand{\sgn}{\operatorname{sgn}}
\newcommand{\Ad}{\operatorname{Ad}}
\newcommand{\ad}{\operatorname{ad}}
\newcommand{\orb}{\operatorname{orb}}

\def\Re{{\operatorname{Re}}}
\def\det{{{\operatorname{det}}}}
\def\calK{{\mathcal{K}}}

\newcommand{\N}{\mathbb{N}}
\newcommand{\C}{\mathbb{C}}
\newcommand{\R}{\mathbb{R}}
\newcommand{\Rp}{{\mathbb{R}}^*_+}
\newcommand{\T}{\mathcal{T}}

\newcommand{\Z}{\mathbb{Z}}
\newcommand{\Zp}{{\mathbb{Z}}_+}
\def\AF{{{\operatorname{AF}}}}
\def\K{{{\operatorname{K}}}}
\def\Ext{{{\operatorname{Ext}}}}
\def\Exts{{{\operatorname{Ext}_{\operatorname{s}}}}}
\def\Extw{{{\operatorname{Ext}_{\operatorname{w}}}}}
\def\Ext{{{\operatorname{Ext}}}}
\def\KK{{{\operatorname{KK}}}}

\def\OA{{{\mathcal{O}}_A}}
\def\ON{{{\mathcal{O}}_N}}

\def\TA{{{\mathcal{T}}_A}}
\def\TAT{{{\mathcal{T}}_{A^t}}}
\def\DTAT{{{\mathcal{D}(\mathcal{T}_{A^t})}}}
\def\TB{{{\mathcal{T}}_B}}
\def\TBT{{{\mathcal{T}}_{B^t}}}

\def\TN{{{\mathcal{T}}_N}}

\def\A{{\mathcal{A}}}

\def\B{{\mathcal{B}}}
\def\E{{\mathcal{E}}}
\def\Q{{\mathcal{Q}}}
\def\calQ{{\mathcal{Q}}}

\def\OB{{{\mathcal{O}}_B}}
\def\OAT{{{\mathcal{O}}_{A^t}}}
\def\OBT{{{\mathcal{O}}_{B^t}}}
\def\F{{\mathcal{F}}}
\def\G{{\mathcal{G}}}
\def\FA{{{\mathcal{F}}_A}}
\def\PA{{{\mathcal{P}}_A}}
\def\C{{\mathbb{C}}}

 \def\U{{\mathcal{U}}}
\def\OF{{{\mathcal{O}}_F}}
\def\DF{{{\mathcal{D}}_F}}
\def\FB{{{\mathcal{F}}_B}}
\def\DA{{{\mathcal{D}}_A}}
\def\DB{{{\mathcal{D}}_B}}
\def\DZ{{{\mathcal{D}}_Z}}

\def\Ext{{{\operatorname{Ext}}}}
\def\Max{{{\operatorname{Max}}}}
\def\Max{{{\operatorname{Max}}}}
\def\max{{{\operatorname{max}}}}
\def\KMS{{{\operatorname{KMS}}}}
\def\Proj{{{\operatorname{Proj}}}}
\def\Out{{{\operatorname{Out}}}}
\def\Aut{{{\operatorname{Aut}}}}
\def\Ad{{{\operatorname{Ad}}}}
\def\Inn{{{\operatorname{Inn}}}}
\def\Int{{{\operatorname{Int}}}}
\def\det{{{\operatorname{det}}}}
\def\exp{{{\operatorname{exp}}}}
\def\nep{{{\operatorname{nep}}}}
\def\sgn{{{\operatorname{sign}}}}
\def\cobdy{{{\operatorname{cobdy}}}}

\def\Ker{{{\operatorname{Ker}}}}
\def\Coker{{{\operatorname{Coker}}}}
\def\Im{{\operatorname{Im}}}
\def\Hom{{\operatorname{Hom}}}

\def\Ew{{{\operatorname{Ext_w}}}}
\def\Es{{{\operatorname{Ext_s}}}}

\def\ind{{{\operatorname{ind}}}}
\def\Ind{{{\operatorname{Ind}}}}
\def\id{{{\operatorname{id}}}}
\def\supp{{{\operatorname{supp}}}}
\def\co{{{\operatorname{co}}}}

\def\I{{\mathcal{I}}}
\def\Span{{{\operatorname{Span}}}}
\def\event{{{\operatorname{event}}}}
\def\S{\mathcal{S}}
\def\calO{\mathcal{O}}
\def\calT{\mathcal{T}}
\def\D{\frak{D}}

%%%%%%%%%%%%%%%%%%%%%%%%%%%%%%%%%%%%%%%%%%%%%%%
%%%%%%%%%%%%%%%%%%%%%%%%%%%%%%%%%%%%%%%%%%
\section{Introduction}
%%%%%%%%%%%%%%%%%%%%%%%%%%%%%%%%%%%%%%

Let $\A$ be a separable unital $C^*$-algerba.
Let us denote by $\calK(H)$ the $C^*$-algebra of compact operators on a separable 
infinite dimensional Hilbert space $H$.
An extension of a $C^*$-algebra $\A$ means a short exact sequence
\begin{equation} \label{eq:extA}
0 \longrightarrow \calK(H) 
  \longrightarrow \E 
 \longrightarrow \A
  \longrightarrow 0
\end{equation}
of $C^*$-algebras  
for which a unital $C^*$-algebra $\E$ contains $\calK(H)$ as an essential ideal.
Classification of extensions of $C^*$-algebras have been developing from many sides, 
K-theory of $C^*$-algebras, non commutative geometry, 
structure theory of $C^*$-algebras, operator theory,
 homotopy theory in algebraic topology, etc. 
 (see text books \cite{Blackadar}, \cite{Douglas}, \cite{HR}, etc.).
 In this paper, we will study 'duality' of extensions of $C^*$-algebras
 from the viewpoints of K-homology theory and K-group theory of $C^*$-algebras.

The notion of K-theoretic 'dual' of a $C^*$-algebra was first introduced
 by W. L. Paschke in \cite{Paschke1981}
from the viewpoint of a duality of $\Ext$-groups and $\K$-groups.
After Paschke, N. Higson refined the dual $C^*$-algebra written $\D(\A)$  of a $C^*$-algebra $\A$
to define $\K$-homology groups $\K^*(\A)$ for $\A$ (\cite{Higson}, cf. \cite{HR}).
By using the $\K$-homology groups, 
Higson and Higson--Roe in \cite{HR}  studied  analytic K-homology theory for $C^*$-algebras comparing 
Kasparov's $\KK$-theory (\cite{Kasparov1975}, \cite{Kasparov1981}). 
  Kaminker--Putnam in \cite{KP} introduced 
  the notion of $\K$-theoretic duality
 from the viewpoint of Spanier-Whitehead dual  
  in topology  and showed  that the Cuntz--Krieger algebras $\OA$ and $\OAT$ 
 are the $\K$-theoretic duality pair (cf. \cite{KPW},\cite{KS}).

In this paper, we introduce the notion of K-theoretic duality 
for extensions of a separable unital nuclear $C^*$-algebras.
Let us denote by $\widetilde{\K}^*(\A)$ and ${\K}^*(\A)$ the reduced K-homology 
and the unreduced K-homology of a $C^*$-algebra $\A$ as in Higson--Roe's text book 
\cite{HR}.
Since
$\widetilde{\K}^1(\A)$ and ${\K}^1(\A)$
are identified with 
the strong extension group $\Es(\A)$ and 
the weak extension group $\Ew(\A)$
for a separable unital nuclear $C^*$-algebra 
$\A$, 
the extension \eqref{eq:extA} of $\A$ 
gives rise to an element $[\E]_s$ in $\widetilde{\K}^1(\A)$ 
and an element $[\E]_w$ in ${\K}^1(\A)$, respectively.
Two extensions 
\begin{equation} \label{eq:extAandAprime}
0 \longrightarrow \calK(H) 
  \longrightarrow \E 
 \longrightarrow \A
  \longrightarrow 0, \qquad
0 \longrightarrow \calK(H) 
  \longrightarrow \E^\prime 
 \longrightarrow \A^\prime
  \longrightarrow 0
\end{equation}
are said to be a K-{\it theoretic duality pair}\/
(Definition \ref{def:Kduality})
if the K-homology long exact sequence for $\A$ (resp. $\A'$)
\begin{align*}
\cdots
& \longrightarrow \widetilde{\K}^p(\A)
\longrightarrow {\K}^p(\A)
\longrightarrow {\K}^p(\mathbb{C})
\longrightarrow \widetilde{\K}^{p+1}(\A)
\longrightarrow {\K}^{p+1}(\A)
\longrightarrow \cdots \\
(\text{resp.}
\cdots
& \longrightarrow \widetilde{\K}^p(\A')
\longrightarrow {\K}^p(\A')
\longrightarrow {\K}^p(\mathbb{C})
\longrightarrow \widetilde{\K}^{p+1}(\A')
\longrightarrow {\K}^{p+1}(\A')
\longrightarrow \cdots )
\end{align*}
is realized as the K-theory long exact sequence for the other extension
\begin{align*}
\cdots
&\longrightarrow {\K}_{p+1}(\E')
\longrightarrow {\K}_{p+1}(\A')
\longrightarrow {\K}_p(\calK(H))
\longrightarrow \widetilde{\K}_{p}(\E')
\longrightarrow {\K}_{p}(\A')
\longrightarrow \cdots \\
(\text{resp.}
\cdots
& \longrightarrow {\K}_{p+1}(\E)
\longrightarrow {\K}_{p+1}(\A)
\longrightarrow {\K}_p(\calK(H))
\longrightarrow \widetilde{\K}_{p}(\E)
\longrightarrow {\K}_{p}(\A)
\longrightarrow \cdots)
\end{align*}
for $p=0,1$ (mod 2).
If in particular, we may find the above realizations 
between $\K^0(\mathbb{C})$ and 
$\K_0(\calK(H))$ in order  preserving way, 
and  
between $\widetilde{\K}^1(\A)$ and $\K_0(\E')$
(resp. between $\widetilde{\K}^1(\A')$ and $\K_0(\E)$)
in sending  $[\E]_s$ to $- [1_{\E'}]$
(resp. $[\E']_s$ to $- [1_{\E}]$),
then 
 the K-theoretic duality pair 
 is said to be {\it strong}\/ K-{\it theoretic duality pair}\/
 (Definition \ref{def:Kduality}), 
 where $ [1_{\E'}], [1_\E]$ denote the classes in $\K_0(\E'), \K_0(\E)$ of the units of 
 $\E', \E$, respectively.

 For a Cuntz--Krieger algebra $\OA$, there exists a specific extension 
 \begin{equation}\label{eq:Toeplits}
\tau_A:
\qquad
0 \longrightarrow \calK(H) 
\overset{\iota_A}{\longrightarrow} \TA 
\overset{\pi_A}{\longrightarrow} \OA
  \longrightarrow 0
 \end{equation}
  called the Toeplitz extension of $\OA$ that was introduced in
  \cite{EFW1979} and \cite{Ev} (cf. \cite{KP}).
 We will prove that the K-homology long exact sequence for $\OA$ is realized as the 
 K-theory cyclic six term exact sequence
  for the Toeplitz extension $\tau_{A^t}$ 
 defined by the the transposed matrix $A^t$ of $A$
 such that 
 the class $[\TA]_s \in \widetilde{\K}(\OA)$ 
 of the Toeplitz extension $\tau_A$ exactly
 corresponds to the class $-[1_{\TAT}] \in \K_0(\TAT)$
 of the minus of the unit $1_{\TAT}$ of $\TAT$ 
 under the identification between $\K^0(\mathbb{C})$ and $\K_0(\calK(H))$
 in order preserving way (Theorem \ref{thm:OATAT}).
 We then know that the Toeplitz extension of a Cuntz-Krieger algebra 
is the K-theoretic dual of the Toeplitz extension of the Cuntz-Krieger algebra for the
transposed matrix in the following way. 
\begin{theorem}\label{thm:main2}
Let $A$ be an irreducible non permutation matrix with entries in $\{0,1\}$.
The Toeplitz extensions $\tau_A$ and $\tau_{A^t}$
of the Cuntz--Krieger algebras 
$\OA$ and $\OAT$ are strong K-theoretic duality pair. 
\end{theorem}

 In general an isomorphic pair 
 $\OA\cong \OB$ of Cuntz--Krieger algebras does not necessarily yield an isomorphic pair
 $\OAT\cong \OBT$ of the Cuntz--Krieger algebras of their transposed matrices. 
 For example, the Cuntz--Krieger algebra $\OA$ for the matrix
$
A=
\left[
\begin{smallmatrix}
1 & 1& 1 \\
1 & 1& 1 \\
1 & 0& 0 
\end{smallmatrix}
\right]
$
 is isomorphic to 
$\mathcal{O}_3$,
whereas
$\OAT$ is not isomorphic to $\mathcal{O}_3$ (cf. \cite{EFW1981}, \cite{Ro}).
Eilers--Restroff--Ruiz in \cite{ERR2013} (cf. \cite{ERRS})
showed that certain extensions of classifiable $C^*$-algebras are strongly classified 
by their  associated cyclic six term exact sequences in K-groups
together with the positive cones of the $\K_0$-groups of the ideals and quotients.
By applying their classification result in our setting,  
 as an application of Theorem \ref{thm:main2},
we will prove the following theorem.  
 
\begin{theorem}\label{thm:main3}
Let $A, B$ be irreducible non permutation matrices with entries in $\{0,1\}.$
The following are equivalent.
%Keep the above notation.\hspace{4cm}
 \begin{enumerate}
\renewcommand{\theenumi}{(\roman{enumi})}
\renewcommand{\labelenumi}{\textup{\theenumi}}
\item There exist isomorphisms 
$\Phi:\OA\longrightarrow \OB$ and $\Psi:\TA\longrightarrow \TB$
of $C^*$-algebras such that the diagram 
\begin{equation}\label{eq:CDTATBOAOB}
\begin{CD}
\TA @>{\pi_A}>> \OA \\
@V{\Psi}VV @VV{\Phi}V \\ 
\TB @>{\pi_B}>> \OB 
\end{CD}
\end{equation}
commutes.
\item There exist isomorphisms 
$\Phi^t:\OAT\longrightarrow \OBT$ and $\Psi^t:\TAT\longrightarrow \TBT$
of $C^*$-algebras such that the diagram 
\begin{equation} \label{eq:CDTATBOAOB2}
\begin{CD}
\TAT @>{\pi_{A^t}}>> \OAT \\
@V{\Psi^t}VV @VV{\Phi^t}V \\ 
\TBT @>{\pi_{B^t}}>> \OBT 
\end{CD}
\end{equation}
commutes.
\end{enumerate}
\end{theorem}
Since an isomorphism from $\TA$ to $\TB$ automatically
induces the commutative diagram in Theorem \ref{thm:main3} (i),
we have the following corollary.
\begin{corollary}%[{Corollary \ref{cor:TATB}}]
\label{cor:main4}
Let $A, B$ be irreducible non permutation matrices with entries in $\{0,1\}.$
the Toeplitz algebras 
$\TA$ and $\TB$ are isomorphic if and only if 
the Toeplitz algebras
$\TAT$ and $\TBT$ defined by their transposed matrices are isomorphic.
\end{corollary}
We will finally present a classification result for the Toeplitz algebras
(Proposition \ref{prop:classification TA})
with several examples (Section \ref{sect:Examples}).

%\medskip

%The contents of the paper is the following:
%\begin{center}
%Contents
%\end{center}

%\begin{enumerate}
%\renewcommand{\theenumi}{\arabic{enumi}}
%\renewcommand{\labelenumi}{\textup{\theenumi}}
%\item 

%1. Introduction

%\item 

%2. Preliminaries for $\Ext$-groups and $\K$-homology groups
%\item 

%3. Cuntz--Krieger algebras and  $\K$-homology groups
%\item 

%4. K-theory for Toeplitz algebra $\TAT$
%\item 

%5. K-theoretic duality for Toeplitz extensions
%\item 

%6. K-theoretic duality for extensions
%\item 

%7. Duality for Toeplitz extensions
%\item

%8. Examples
%\end{enumerate}

 %%%%%%%%%%%%%%%%%%%%%%%%%%%%%%%%%%%%%%%%%%%%%%%%%%%%%%%%%
%%%%%%%%%%%%%%%%%%%%%%%%%%%%%%%%%%%%%%%%%%
\section{Preliminaries for $\Ext$-groups and $\K$-homology groups}
%%%%%%%%%%%%%%%%%%%%%%%%%%%%%%%%%%%%%%
In this section, we briefly give an introduction of $\K$-homology groups based on \cite{Higson} and \cite{HR},
see \cite{HR} for detail.
In what follows, $\A$ stands for a separable unital nuclear  $C^*$-algebra.
%%%%%%%%%%%%%%%%%%%%%%%%%%%%%%%%%%%%%%%%%%%%%%
\subsection{$\Ext$-groups}
%%%%%%%%%%%%%%%%%%%%%%%%%%%%%%%%%%%%%%%%%%%
It is well-known that an  extension \eqref{eq:extA} bijectively corresponds 
to a unital $*$-monomorphism $\tau:\A\longrightarrow \Q(H)$ 
that is called a Busby invariant,
where $\Q(H)$
denotes the Calkin algebra $\B(H)/\calK(H)$.
We also call such a $*$-monomorphism $\tau:\A\longrightarrow \Q(H)$
an extensionof $\A$.
Two extensions
$\tau_i : \A \longrightarrow \Q(H), i=1,2$
are said to be 
weak equivalent if there exists a unitary $u\in \Q(H)$ 
such that 
$\tau_2(a) = u \tau_1(a) u^*, a \in \A$. 
If in particular the above unitary is taken as $u=\pi(U)$ 
for some unitary $U \in \B(H)$,  
where $\pi:\B(H)\longrightarrow \Q(H)$ denotes the natural quotient map,
  then the extensions
$\tau_i : \A \longrightarrow \Q(H), i=1,2$
are said to be  
strong equivalent.
Let us denote by $\Es(\A)$ and $\Ew(\A)$ the strong equivalence classes 
and the weak equivalence classes of extensins of $\A$, respectively.
The class of a unital $*$-monomorphism
$\tau:\A\longrightarrow \Q(H)$ in $\Es(\A)$ is denoted by $[\tau]_s$, and similarly 
$[\tau]_w$ in $\Ew(\A).$  
Fix an identification $H \oplus H = H$ which induces an embbeding
$\Q(H) \oplus \Q(H) \hookrightarrow \Q(H)$.
The sum of extensions 
$\tau_1 \oplus \tau_2$ are defined by a direct sum $\tau_1\oplus \tau_2$ through the embedding  
$\Q(H) \oplus \Q(H) \hookrightarrow \Q(H)$.
It is well-known that both $\Es(\A)$ and $\Ew(\A)$ become abelian  semigroups,
and also they are abelian groups for nuclear $C^*$-algebra $\A$ 
(cf. \cite{Arveson}, \cite{BDF}, \cite{CE}, \cite{Douglas}, etc.).
They are called the strong extension group for $\A$ 
and the weak extension group for $\A$.
Let us denote by $q_A: \Es(\A) \longrightarrow \Ew(\A)$
the surjective homomorphism induced by a natural quotient map.
As in \cite{HR} and \cite{PP}, there exists a homomorphism
$\iota_\A: \Z \longrightarrow \Es(\A)$
such that the sequence
\begin{equation}\label{eq:ZSW}
\Z
 \overset{\iota_\A}{\longrightarrow} \Es(\A)
 \overset{q_\A}{\longrightarrow} \Ew(\A)
 %\longrightarrow 0
\end{equation}
is exact at the middle, and hence
the quotient group $\Es(\A)/\iota_\A(\Z)$ is isomorphic to $\Ew(\A).$

%%%%%%%%%%%%%%%%%%%%%%%%%%%%%%%%%%%%%%%%%%%%%%
\subsection{$\K$-homology groups}
%%%%%%%%%%%%%%%%%%%%%%%%%%%%%%%%%%%%%%%%%%%
Let $\A$ be a separable unital $C^*$-algebra.
Following \cite{HR}, a non degenerate representation $\rho: \A\longrightarrow \B(H)$
on a separable infinite dimensional Hilbert space $H$ is said to be {\it ample}\/
if $\rho(a)$ never belongs to $\calK(H)$ unless $a=0$. 
   Hence any non degenerate representation of a simple $C^*$-algebra 
is always ample if $\A$ is of infinite dimensional.  
Let $\rho:\A\longrightarrow \B(H)$ be an ample representation.
According to 
Paschke \cite{Paschke1981}, Higson \cite{Higson} and Higson--Roe \cite{HR}, 
the dual $C^*$-algebra of $\A$ 
means the $C^*$-subalgebra $\D_\rho(\A)$  of $\B(H)$ defined by
\begin{equation}
\D_\rho(\A) = \{T \in \B(H) \mid [T,\rho(a)] \in \calK(H) \text{ for all } a \in \A\}.
\end{equation} 
The isomorphism type of the $C^*$-algebra $\D_\rho(\A)$ does not depend on the choice of ample 
representations, so that it is denoted by $\D(\A)$
(\cite{Higson}, \cite{HR}, \cite{Paschke1981}).
Let us denote by $\widetilde{A}$ the unitization of $\A$ which contains 
$\A$ as a maximal two-sided ideal with codimension one.
In 
Higson \cite{Higson} and Higson--Roe \cite{HR},
the reduced $\K$-homology groups $\widetilde{\K}^p(\A)$
and
the unreduced $\K$-homology groups $\K^p(\A)$ for $p=0,1$ are defined by
\begin{equation*}
\widetilde{\K}^p(\A) := \K_{1-p}(\D(\A))
\quad
\text{ and }
 \quad 
\K^p(\A) := \K_{1-p}(\D(\widetilde{\A}))
\quad
\text{ for } p= 0,1.% \quad (\text{see } \cite{Higson}, \cite{HR}).
\end{equation*}
\begin{lemma}[{\cite{HR}}] \label{lem:HR2.1}
For a separable unital nuclear $C^*$-algebra $\A$, we have
 \begin{equation*}
\widetilde{\K}^1(\A) %=&  \K_0(\D(\A)) 
= \Ext_s(\A), \qquad
{\K}^1(\A)  %= & \K_0(\D(\widetilde{\A})) 
= \Ext_s(\widetilde{\A})= \Ext_w(\A).
\end{equation*}
\end{lemma}
%The abelian group $\K^0(\mathbb{C})$ 
%is isomorphic to the integer group $\Z$. 
%Now we choose a specific isomorphism 
%$\K^0(\mathbb{C}) \cong \Z$ in the following way.
%By definition of the group $\K^0(\mathbb{C})$, we know that 
As in \cite[Section 5]{HR}, the group
$\K_*(\D(\widetilde{\mathbb{C}}))$ is naturally identified with 
$\K_*(\D(\mathbb{C})/\calK(H))$.
Since $\D(\mathbb{C})$ is $\B(H)$,
$\K^0(\mathbb{C})$
%$\K_1(\D(\widetilde{\mathbb{C}}))$
is identified with 
$\K_1(\Q(H))$.
An element of $\K_1(\Q(H))$ is the class $[U]$ of
a unitary $U \in \Q(H)$ whose Fredholm index $\Ind(U)$ takes its value in 
$\Z$. 
This gives an isomorphism
between $\K^0(\mathbb{C})$ and $\Z$.
We fix this isomorphism and identify the group
$\K^0(\mathbb{C})$ with $\Z$ to write 
$\K^0(\mathbb{C})=\Z$.
By Lemma \ref{lem:HR2.1},
%
%Under the identifications
%$\widetilde{\K}^1(\A) = \Ext_s(\A), {\K}^1(\A) =  \Ext_w(\A)$
%and$\K^0(\mathbb{C}) = \Z$,
the sequence \eqref{eq:ZSW} is rephrased as
\begin{equation*}%\label{eq:ZSW}
\K^0(\mathbb{C})
 \overset{\iota_\A}{\longrightarrow} \widetilde{\K}^1(\A)
 \overset{j_\A^1}{\longrightarrow} {\K}^1(\A)
%\longrightarrow 0
\end{equation*}
that is exact at the middle,
where the map 
$q_\A:\Es(\A) \longrightarrow \Ew(\A)$ in \eqref{eq:ZSW} 
is written as $j_\A^1$.
As $\K^0$ is a contravariant functor,
 the natural inclusion 
$\mathbb{C} \hookrightarrow \A$ %of $\mathbb{C}$ to the scalar multiples of the unit in $\A$
induces a homomorphism
$\iota_{\mathbb{C}}^*:\K^0(\A) \longrightarrow \K^0(\mathbb{C})$.
As in \cite[(5.2.8)]{HR},
the algebra $\D(\widetilde{\A})$ is identified with 
$\begin{bmatrix}
\D(\A) & \calK(H) \\
\calK(H) & \B(H)
\end{bmatrix}.
$
The inclusion of $\D(\A)$ into the upper left corner of $\D(\widetilde{\A})$
  induces a homomorphism from
  $\widetilde{\K}^0(\A)$ to $\K^0(\A)$ that we denote by $j_\A^0$.
 We then have the following K-homology long exact sequence.
\begin{lemma}[{\cite{HR}}] 
For  a separable unital $C^*$-algebra $\A$, 
the $\K$-homology  cyclic six term exact sequence
\begin{equation} \label{eq:Khomolongexact}
\begin{CD}
0 @>>> \widetilde{\K}^0(\A) @>{j_A^0}>> \K^0(\A) \\
@AAA  @.   @VV{\iota_{\mathbb{C}}^*}V  \\
 \K^1(\A)  @<{j_\A^1}<< \widetilde{\K}^1(\A) @<{\iota_\A}<< \K^0(\mathbb{C}) 
\end{CD}
\end{equation}
holds.
\end{lemma}
For the group
$\K_0(\calK(H))$,
the correspondence from the class $[e]$ of a projection $e$ of rank one in 
$\calK(H)$ to the positive generator $1$ of the integer group $\Z$
induces an isomorphism between 
$\K_0(\calK(H))$ and $\Z$.
We fix this correspondence %, we may identify 
%the groups  $\K_0(\calK(H))$ and $\Z$,
and write $\K_0(\calK(H))=\Z$.
%The preceding identification $\K^0(\mathbb{C})=\Z$ 
%together with this identification $\K_0(\calK(H))=\Z$ allow us to write
%$\K^0(\mathbb{C}) =\K_0(\calK(H)) =\Z.$

%%%%%%%%%%%%%%%%%%%%%%%%%%%%%%%%%%%%%%%%%%
\section{Cuntz--Krieger algebras and  $\K$-homology groups}
%%%%%%%%%%%%%%%%%%%%%%%%%%%%%%%%%%%%%%
%%%%%%%%%%%%%%%%%%%%%%%%%%%%%%%%%%%%%%%%%%%%%%
\subsection{Cuntz--Krieger algebras and Toeplitz algebras}
%%%%%%%%%%%%%%%%%%%%%%%%%%%%%%%%%%%%%%%%%%%
 Let $A = [A(i,j)]_{i,j=1}^N$ be an irreducible non permutation matrix with entries in $\{0,1\}.$
 The Cuntz--Krieger algebra $\OA$ for the matrix $A$ is defined by the universal $C^*$-algebra 
 generated by partial isometries $S_1,\dots,S_N$  
 subject to the operator relations:
 \begin{equation}\label{eq:CK}
 \sum_{j=1}^N S_j S_j^* =1,\qquad S_i^* S_i = \sum_{j=1}^N A(i,j) S_j S_j^*, \quad i=1,\dots, N \quad (\cite{CK}).
 \end{equation}
 It is a separable unital nuclear simple purely infinite $C^*$-algebra
 satisfying UCT. 
Let $\mathbb{C}^N$ be the $N$-dimensional Hilbert space with orthonormal basis
$\{\xi_1,\dots,\xi_N\}$.
Let $H_0$ be the one-dimensional Hilbert space with unit vector $v_0$.
The full Fock space
$F_N$ is defined by 
the Hilbert space 
$H_0 \oplus (\oplus_{n=1}^\infty H^{\otimes n})$,
where
$H^{\otimes n}$ is the $n$-fold tensor product 
$\mathbb{C}^N \otimes\cdots \otimes\mathbb{C}^N$.
The sub Fock space $H_A$ associated to $A$ is a subspace of $F_N$
defined by the closed linear span of vectors
$$
\{v_0\} \cup \{ \xi_{k_1}\otimes\cdots\otimes\xi_{k_n} \in H^{\otimes n} \mid
A(k_j, k_{j+1}) = 1, \, j=1,\dots,n-1, \, n=1,2,\dots \}.
$$
Let us define the creation operators $T_1, \dots, T_N$ on $H_A$ by
\begin{equation*}
T_i v_0 = \xi_i, \qquad
T_i (\xi_{k_1}\otimes\cdots\otimes\xi_{k_n}) = 
\begin{cases}
\xi_i\otimes \xi_{k_1}\otimes\cdots\otimes\xi_{k_n} & \text{ if } A(i, k_1) =1, \\
0 & \text{ otherwise. }
\end{cases}
\end{equation*}  
Let us denote by $P_0$ the projection of rank one on $H_A$ onto $H_0$.
The creation operators $T_1,\dots, T_N$ are partial isometries on $H_A$ 
satisfying the operator relations:
\begin{equation}\label{eq:TA}
\sum_{j=1}^N T_j T_j^* + P_0 = 1,\qquad
T_i^* T_i = \sum_{j=1}^N A(i,j) T_j T_j^* + P_0, \quad i=1,\dots,N
\quad
(\cite{EFW1981},\cite{Ev}).
\end{equation}
The $C^*$-algebra $C^*(T_1,\dots, T_N, P_0)$ generated by 
$T_1,\dots, T_N, P_0$ is called the Toeplitz algebra for $A$, that is denoted by $\TA$.
Let 
$S_i, i=1,\dots,N$
be the generating partial isometries of $\OA$ satisfying \eqref{eq:CK}.
%that generate the Cuntz--Krieger algebra $\OA$.
The correspondence
$T_i \longrightarrow S_i \quad \text{ and } \quad P_0 \longrightarrow 0$
yields a surjective $*$-homomorphism
$\TA \longrightarrow \OA$ which is written $\pi_A$.
The kernel of the map $\pi_A: \TA \longrightarrow \OA$ 
is generated by the operators of the form
$T_{\mu_1}\cdots T_{\mu_m} P_0 T_{\nu_n}^*\cdots T_{\nu_1}^*,
\mu_1,\dots,\mu_m, \nu_1,\dots,\nu_n \in \{1,\dots,N\}
$
that generates  a $C^*$-algebra isomorphic to $\calK(H_A)$.
Hence we have an extension
\begin{equation}\label{eq:Toeplitzext}
\tau_{A}: \quad 0 \longrightarrow \calK(H_A) \longrightarrow \TA \longrightarrow \OA \longrightarrow 0
\end{equation}
called the Toeplitz extension for $\OA$.
It defines an element 
$[\TA]_s $ in $\Es(\OA)$ and 
$[\TA]_w $ in $\Ew(\OA)$, respectively.

  %%%%%%%%%%%%%%%%%%%%%%%%%%%%%%%%%

 %%%%%%%%%%%%%%%%%%%%%%%%%%%%%%%%%%%%%%%%%%
\subsection{Cyclic exact sequence of $\Ext_*(\OA)$-groups}
%%%%%%%%%%%%%%%%%%%%%%%%%%%%%%%%%%%%%%
% In \cite{CK}, Cuntz--Krieger computed the weak extension group 
 %$\Ew(\OA) $ of $\OA$ to be $ \Z^N / (I - A)\Z^N$.
%Their sterategy to compute it is the following.
Let $\sigma: \OA\longrightarrow \Q(H)$ be an extension.
Put $e_i = \sigma(S_i S_i^*), i=1,\dots N$.
Take a $*$-monomorphism $\tau_0:\OA\longrightarrow \B(H)$ 
such that $\pi(\tau_0(S_i S_i^*)) = e_i$ 
and consider the extension
$\tau = \pi\circ \tau_0:\OA\longrightarrow \Q(H)$.
Since the partial isometry $\sigma(S_i) \tau(S_i^*)$ commutes with $e_i$, 
its restriction $e_i \sigma(S_i) \tau(S_i^*)e_i$ becomes a unitary in $e_i \Q(H) e_i$,
so we may define  its Fredholm index $\ind_{e_i}\sigma(S_i) \tau(S_i^*)$ denoted by
$d_i(\sigma,\tau) \in \Z, i=1,\dots,N$.
Cuntz--Krieger in \cite{CK} proved that the correspondence
$[\sigma]_w \in \Ew(\OA)\longrightarrow [d_i(\sigma,\tau)]_{i=1}^N \in \Z^N/(I-A)\Z^N$
is well-defined and gives rise to an isomorphism that is denoted by $d_w$,
so that 
$\Ew(\OA)\cong \Z^N/(I-A)\Z^N.$

For the matrix $A$, we define an $N \times N$ matrix $\widehat{A}$ by 
$\widehat{A} = A + R_1 -AR_1$,
 where $R_1$ is the $N \times N$ matrix whose first row is $[1,\dots,1]$,
 and other rows are zero vectors.  
In \cite{MaPre2021exts}, 
it was shown that the above family 
$d_i(\sigma,\tau) \in \Z, i=1,\dots,N$
%of the indices of the Fredholm operators
%$e_i \sigma(S_i) \tau(S_i^*)e_i, i=1,\dots,N$
%defined above 
also yields a homomorphism  
$[\sigma]_s \in \Es(\OA)\longrightarrow [d_i(\sigma,\tau)]_{i=1}^N \in \Z^N/(I-\widehat{A})\Z^N$
that is actually an isomorphism, denoted by $d_s$.
 Since $I - \widehat{A} = (I - A) (I + R_1)$,
  there is a natural quotient map
 from $\Z^N / (I - \widehat{A})\Z^N$ to $\Z^N / (I - A)\Z^N$
 that is denoted by $\hat{q}_A.$

 As in \cite{MaPre2021exts}, a homomorphism 
 $\hat{\iota}_A: \Z \longrightarrow \Z^N / (I - \widehat{A})\Z^N$ is defined 
 by
 $\hat{\iota}_A(m) = [(I- A)[k_i]_{i=1}^N]$
 % \in \Z^N / (I - \widehat{A})\Z^N$
for $m = \sum_{i=1}^N k_i.$
Then the following two lemmas were shown in \cite{MaPre2021exts}.
\begin{lemma}[{\cite{MaPre2021exts}}] \label{lem:theorem3.3}
%\begin{enumerate}
%\renewcommand{\theenumi}{(\roman{enumi})}
%\renewcommand{\labelenumi}{\textup{\theenumi}}
%\item
There exist  isomorphisms
$$
d_w : \Extw(\OA)\longrightarrow \Z^N/(I-A)\Z^N,\qquad
d_s: \Exts(\OA)\longrightarrow \Z^N/(I-\widehat{A})\Z^N
$$
of groups such that the diagram 
\begin{equation}\label{eq:dsdw}
\begin{CD}
\Z @>\iota_A>> \Exts(\OA) @>q_A>> \Extw(\OA)\\
@|   @VV{d_s}V @VV{d_w}V \\
\Z @>>\hat{\iota}_A> \Z^N/ (I-\widehat{A})\Z^N @>>\hat{q}_A> \Z^N/ (I-A)\Z^N
\end{CD}
\end{equation}
commutes, and
\begin{align} 
d_w([\TA]_w) & =  - [1_N] \quad \text{ in } \quad \Z^N/ (I- {A})\Z^N,  \label{eq:dwTA} \\
d_s([\TA]_s) & = - \hat{\iota}_A(1) - [1_N] \quad { in } \quad \Z^N/ (I- \widehat{A})\Z^N, \label{eq:dsTA}
\end{align}
where $[1_N]=[(1,\dots,1)]$ means the class of the vector $(1,\dots,1) \in \Z^N.$
\end{lemma}
Let us denote by
$\Ker(I -A), \Ker(I - \widehat{A})$
the subgroups of $\Z^N$ defined by
the kernels in $\Z^N$ of the matrices $I-A$ 
and of $I - \widehat{A},$
respectively.
%In \cite{MaPre2021exts}, the homomorphisms 
Let
$$
i_1: \Z\longrightarrow \Ker(I - \widehat{A}),\qquad
j_A: \Ker(I - \widehat{A})  \longrightarrow \Ker(I - A),\qquad
s_A: \Ker(I - A)\longrightarrow \Z
$$
be homomorphisms of abelian groups
defined by 
\begin{equation}\label{eq:ijs}
i_1( m ) = (m,0,\dots,0),
\quad
j_A(
[l_i]_{i=1}^N)
=(
- \sum_{i=2}^N l_i, l_2, \dots, l_N),
\quad
s_A(
[l_i]_{i=1}^N ) 
=
\sum_{i=1}^N l_i.
\end{equation}
Since $j_A \circ i_1 =0$, the map
$j_A$ induces a homomorphism
$\Ker(I-\widehat{A})/{i_1(\Z)}\longrightarrow \Ker(I-A)$
that is still denoted by $j_A$.
\begin{lemma}[{\cite{MaPre2021exts}}]
We have the following cyclic six term exact sequence.
\begin{equation}\label{eq:6termA}
\begin{CD}
0 @>>> \Ker(I-\widehat{A})/{i_1(\Z)} @>j_A>> \Ker(I-A) \\
@AAA  @.   @VV{s_A}V  \\
   \Z^N/ (I-A)\Z^N  @<\hat{q}_A<< \Z^N/ (I-\widehat{A})\Z^N @<\hat{\iota}_A<< \Z 
\end{CD}
\end{equation}
\end{lemma}

%%%%%%%%%%%%%%%%%%%%%%%%%%%%%%%%%%%%%%%%%%%%%%%%%%%%%%%
%%%%%%%%%%%%%%%%%%%%%%%%%%%%%%%%%%%%%%%%%%
\subsection{K-homology long exact sequence for $\OA$}
%%%%%%%%%%%%%%%%%%%%%%%%%%%%%%%%%%%%%%
In this subsection, we will compute the K-homology cyclic six term exact sequence 
\eqref{eq:Khomolongexact} for Cuntz--Krieger algebra $\OA$.
%Now let us consider the Cuntz--Krieger algebra $\OA$
%for an irreducible non permutation matrix $A = [A(i,j)]_{i,j=1}^N$
%with entries in $\{0,1\}.$
The following lemma is an easy consequence of basic properties of K-theory of Cuntz--Krieger algebras.
 \begin{lemma}\label{lem:phi0K}
There exists an isomorphism $\phi^0: \K^0(\OA) \longrightarrow \Ker(I-A) $
such that 
the diagram 
\begin{equation}\label{eq:phizeroOAKer}
\begin{CD}
 \K^0(\OA) @>{\phi^0}>>  \Ker(I-A) \\
 @V{\iota^*_{\mathbb{C}}}VV   @V{s_A}VV \\
 \K^0(\mathbb{C}) @ = \Z 
 \end{CD}
 \end{equation}
 commutes,
  where $\iota^*_{\mathbb{C}}: \K^0(\OA) \longrightarrow \K^0(\mathbb{C})$
 is the homomorphism induced by the inclusion $\iota_{\mathbb{C}}: \mathbb{C}\hookrightarrow \OA$,
 and $s_A: \Ker(I -A) \longrightarrow \Z$ is the homomorphism defined in \eqref{eq:ijs}.
 %by  $s_A([l_i]_{i=1}^N) = \sum_{i=1}^N l_i.$ 
\end{lemma}
Let us denote by 
$j_A^0: \widetilde{\K}^0(\OA) \longrightarrow \K^0(\OA)$
the injective homomorphism
$j_{\OA}^0: \widetilde{\K}^0(\OA) \longrightarrow \K^0(\OA)$
in the upper horizontal arrow in \eqref{eq:Khomolongexact} for $\A = \OA$.
As 
%rphisms $j_{\OA}^0: \widetilde{\K}^0(\OA) \longrightarrow \K^0(\OA)$and
$j_A:\Ker(I-\widehat{A})/{i_1(\Z)} \longrightarrow \Ker(I-A) $
is  injective, we have the following lemma.
\begin{lemma}\label{lem:tildephi0}
There exists a unique isomorphism 
$\widetilde{\phi}^0: \widetilde{\K}^0(\OA) \longrightarrow \Ker(I-\widehat{A})/{i_1(\Z)}$
such that 
the diagram 
\begin{equation}\label{eq:3.2}
\begin{CD}
 %0 @. 0 \\
 %@VVV   @VVV \\
 \widetilde{\K}^0(\OA) @>{\widetilde{\phi}^0}>> \Ker(I-\widehat{A})/{i_1(\Z)} \\
 @V{j_A^0}VV   @VV{j_A}V \\
 \K^0(\OA) @>{\phi^0}>>  \Ker(I-A) 
 \end{CD}
 \end{equation}
commutes.
\end{lemma}
%\begin{proof}
%In the diagram \eqref{eq:3.2},
%the two vertical sequences are exact 
%by the commutative diagrams \eqref{eq:Khomolongexact} and \eqref{eq:6termA},
%and  the isomorphism
%$\phi^0 : \K^0(\OA) \longrightarrow \Ker(I-A)$ is already defined
%in Lemma \ref{lem:phi0K}.
%Hence 
%there uniquely exists an isomorphism
%$\widetilde{\phi}^0:\widetilde{\K}^0(\OA) \longrightarrow \Ker(I-\widehat{A})/{i_1(\Z)}$
%which yields the desired commutative diagram.
%\end{proof}
Let us denote by 
$\iota_A: \K^0(\mathbb{C}) (=\Z) \longrightarrow \widetilde{\K}^1(\OA)$
the homomorphism
 $\iota_{\OA}: \K^0(\mathbb{C}) \longrightarrow \widetilde{\K}^1(\OA)$
in the lower horizontal arrow in \eqref{eq:Khomolongexact} for $\A = \OA$.
Since $\widetilde{\K}^1(\OA)= \Ext_s(\OA)$ 
and 
$
\K^1(\OA)=\Ext_w(\OA),
$
by putting 
$$
\tilde{\phi}^1 = d_s : \widetilde{\K}^1(\OA) \longrightarrow \Z^N/ (I-\widehat{A})\Z^N, \qquad
\phi^1 = d_w : \K^1(\OA) \longrightarrow \Z^N/ (I-A)\Z^N, 
$$
 we have the following proposition.
\begin{proposition}\label{prop:Khomocyclic}
The $\K$-homology  cyclic six term exact sequence \eqref{eq:Khomolongexact}
for   $\OA$ 
\begin{equation}\label{eq:sixtermOA}
\begin{CD}
0 @>>> \widetilde{\K}^0(\OA) @>{j_A^0}>> \K^0(\OA) \\
@AAA  @.   @VV{\iota_{\mathbb{C}}^*}V  \\
 \K^1(\OA)  @<{j_A^1}<< \widetilde{\K}^1(\OA) @<{\iota}_A<< \K^0(\mathbb{C}) 
\end{CD}
\end{equation}
is computed to be the cyclic six term exact sequence \eqref{eq:6termA},
such that there is a commutative diagram
\begin{equation}\label{eq:Khomocyclic}
\begin{CD}
 0 @. 0 \\
 @VVV   @VVV \\
 \widetilde{\K}^0(\OA) @>{\widetilde{\phi}^0}>> \Ker(I-\widehat{A})/{i_1(\Z)} \\
 @V{j_A^0}VV   @VV{j_A}V \\
 \K^0(\OA)  @>{\phi^0}>>  \Ker(I-A) \\
 @V{\iota_{\mathbb{C}}^*}VV   @VV{s_A}V \\
 \K^0(\mathbb{C}) @= \Z \\
 @V{\iota_A}VV   @VV{\hat{\iota}_A}V \\
  \widetilde{\K}^1(\OA) @>{\widetilde{\phi}^1}>>   \Z^N/ (I-\widehat{A})\Z^N  \\
 @V{j_A^1}VV   @VV{\hat{q}_A}V \\
  {\K}^1(\OA) @>{\phi^1}>>\Z^N/ (I-A)\Z^N  \\
 @VVV   @VVV \\
  0 @. 0
 \end{CD}
 \end{equation}
 satisfying
 $\widetilde{\phi}^1([\T_A]_s) = -\hat{\iota}_A(1) - [1_N],$
 where
 $[\T_A]_s \in \widetilde{\K}^1(\OA)$ is the strong equivalence class of the Toeplitz extension 
 %which is regarded as an element of $\widetilde{K}^1(\OA)$
 %under the identification between $\Ext_s(\OA)$ and  $\widetilde{K}^1(\OA)$, 
 and $[1_N] =[(1,\dots,1)]$ is the class of the vector
 $(1,\dots,1)$ in $\Z^N$.
\end{proposition}
\begin{proof}
By  Lemma \ref{lem:theorem3.3}, Lemma \ref{lem:phi0K}
and 
Lemma \ref{lem:tildephi0},
we have the commutative diagram \eqref{eq:Khomocyclic}.
By \eqref{eq:dsTA}, we know that 
$\tilde{\phi}^1([\TA]_s) = -\hat{\iota}_A(1) - [1_N].$
\end{proof}
%%%%%%%%%%%%%%%%%%%%%%%%%%%%%%%%%%%%%%%%%%%%%%%%%%%
%{\scriptsize
%\begin{equation*}
%\begin{CD}
% 0 @>>> \widetilde{\K}^0(\OA) @>>> \K^0(\OA) @>>> \K^0(\mathbb{C}) @>>>\widetilde{\K}^1(\OA) @>>> {\K}^1(\OA) @>>> 0  \\
% @. @VV{\widetilde{\Phi}^0}V  @VV{\Phi^0}V @|   @VV{\widetilde{\Phi}^1}V  @VV{\Phi^1}V  @.  \\
%0 @>>>\Ker(I-\widehat{A})/{i_1(\Z)} @>>> \Ker(I-A) @>>> \Z @>>> \Z^N/ (I-\widehat{A})\Z^N  @>>> \Z^N/ (I-A)\Z^N  @>>> 0  \\
%\end{CD}
%\end{equation*}
%}
%%%%%%%%%%%

%%%%%%%%%%%%%%%%%%%%%%%%%%%%%%%%%%%%%%%%%%%%%%%%%
%%%%%%
%%%%%%%%%%%%%%%%%%%%%%%%%%%%%%%%%%%%
\section{K-theory for Toeplitz algebra $\TAT$ }
%%%%%%%%%%%%%%%%%%%%%%%%%%%%%%%%%%%%%%
%\begin{equation*}
%\newcommand{\bg}{%
%\family{cmr}\size{20}{12pt}\selectfont}
%\newcommand{\bigAt}{\smash{\hbox{\bg A^t}}}
%\newcommand{\bigAt}{%
%\smash{\lower1.7ex\hbox{\bg A^t}}}
%\begin{bmatrix}
%0 & \dots & 0 \\
%1 & \bigAt & 
%\end{bmatrix}
%\end{equation*}
In this section, we will compute the K-theory cyclic six term exact sequence for the Toeplitz 
extension \eqref{eq:Toeplitzext} for the transposed matrix $A^t$ of $A$.
Let us consider the Toeplitz algebra $\TAT$  for the transposed matrix $A^t$.
Its generating partial isometries and projection are denoted by
$t_1,\dots, t_N$ and $p_0$ that satisfy 
 the operator relations
\begin{equation}\label{eq:TATrelation}
\sum_{j=1}^N t_j t_j^*  + p_0 = 1,\qquad 
t_i^* t_i = \sum_{j=1}^N A^t(i,j) t_j t_j^* + p_0
\quad \text{ for } i=1,\dots,N \quad (\text{see } \eqref{eq:TA}).
\end{equation}
Let $A_{\T}^t$ be the $(N+1) \times (N+1)$ matrix  defined by 
$
A_{\T}^t
=
\left[
\begin{smallmatrix}
0      &  & \cdots & 0      \\
1      &  &       &        \\ 
\vdots &  & A^t   &        \\ 
1      &  &       &  
\end{smallmatrix}
\right].
$
%\begin{equation*}
%$
%\end{equation*}
%
%\begin{bmatrix}
%0      & 0      & 0      & 0      & 0      \\
%1      & A(1,1) & A(2,1) & \cdots & A(N,1) \\ 
%1      & A(1,2) & A(2,2) & \cdots & A(N,2) \\ 
%\vdots & \vdots &        &        & \vdots \\ 
%1      & A(1,N) & A(2,N) & \cdots & A(N,N) 
%\end{bmatrix}.
%
The $C^*$-algbera $\T_{A^t}$ 
is then regarded as the graph algebra $C^*(G_{A_{\T}^t})$
for the directed graph $G_{A_{\T}^t}$
associated to the matrix 
$A_{\T}^t$.
The graph is obtained by adjoining a sink to the original graph 
$G_{A^t}$ for the matrix $A^t$.
Let $A_1$ be the $(N+1)\times (N+1)$ matrix defined by
$
A_1
=\left[
\begin{smallmatrix}
1     & 1  & \cdots   & 1 \\
0     &    &          & \\
\vdots&    & A        &  \\
0          &          &        
\end{smallmatrix}
\right].
$
The general K-theory formula due to Raeburn--Szyma\'nski \cite{RS} for graph $C^*$-algebras tells us 
the following lemma.
\begin{lemma}[{Raeburn--Szyma\'nski \cite{RS}}]\label{lem:RS2}
Keep the above notation.\hspace{4cm}
 \begin{enumerate}
\renewcommand{\theenumi}{(\roman{enumi})}
\renewcommand{\labelenumi}{\textup{\theenumi}}
\item
$ \K_0(\T_{A^t})
=\Z^{N+1}/(I - A_1)\Z^{N+1}.$
\item
$\K_1(\T_{A^t})
= \Ker(s_A: \Ker(I-A) \longrightarrow \Z).$
\end{enumerate}
\end{lemma}
%\begin{proof}
%\end{proof}
Let us define homomorphisms
\begin{gather*}
\iota_{A_1}: \Z \longrightarrow  \Z^{N+1}(I - A_1)\Z^{N+1}
 \quad \text{and}
 \quad
q_{A_1}: \Z^{N+1}/(I - A_1)\Z^{N+1} \longrightarrow \Z^N/(I-A)\Z^N
\end{gather*}
by setting
$
\iota_{A_1}(m) = [0, \hat{\iota}_A(m)]
$ 
for 
$ m \in \Z$ and
$
q_{A_1}([(z_0,z_1,\dots,z_N)]) 
= [(z_1,\dots,z_N)] 
$
for
$ [(z_0,z_1,\dots,z_N)]\in \Z^{N+1}/(I - A_1)\Z^{N+1}.
$
%It is straightforward to see that the maps give well-defined homomorphisms.
Let
$$
\iota_{s_A}:\Ker(s_A: \Ker(I-A)\longrightarrow \Z) \longrightarrow
 \Ker(I-A)
 $$
be the natural inclusion map.
% defined by
% $\iota_{s_A}([x_i]_{i=1}^N) =[x_i]_{i=1}^N.$
It is direct to see the following lemma.
\begin{lemma}
We have the following cyclic six term  exact sequence.
\begin{equation}\label{eq:sixtermA1}
\begin{CD}
0 @>>> \Ker(s_A: \Ker(I-A)\longrightarrow \Z) @>{\iota_{s_A}}>> \Ker(I-A) \\
@AAA  @.   @VV{s_A}V  \\
   \Z^N/ (I-A)\Z^N  @<{q_{A_1}}<< \Z^{N+1}/ (I-{A_1})\Z^{N+1} @<\iota_{A_1}<< \Z 
\end{CD}
\end{equation}
\end{lemma}
Let  
$s_i, i=1,\dots,N$
be the generating partial isometries of the Cuntz--Krieger algebra 
$\OAT$ for the transposed matrix $A^t$, that satisfy the operator relations
\begin{equation}\label{eq:CKAT}
\sum_{j=1}^N  s_j s_j^{*}   = 1,\qquad 
s_i^{*} s_i = \sum_{j=1}^N A^t(i,j) s_j s_j^{*} 
\quad \text{ for } i=1,\dots,N \quad (\text{see } \eqref{eq:CK}).
\end{equation}
As in \eqref{eq:Toeplitzext},
the correspondence
$t_i\longrightarrow s_i$ and $p_0 \longrightarrow 0$
yields a surjective $*$-homomorphism
$\TAT \longrightarrow \OAT$ which is written $\pi_{A^t}$,
so that we have an extension 
%The kernel of the map $\pi_{A^t}: \T_{A^t} \longrightarrow \OAT$ 
%is isomorphic to $\calK(H)$.
%Hence we have an extension
\begin{equation}\label{eq:transposeToeplitaext}
\tau_{A^t}: \quad 0 \longrightarrow \calK(H) \longrightarrow \T_{A^t}  
\overset{\pi_{A^t}}{\longrightarrow} \mathcal{O}_{A^t} \longrightarrow 0.
\end{equation}
\begin{lemma}\label{lem:etaTO}
Define homomorphisms
\begin{equation*}
\eta^\T_0: \K_0(\T_{A^t}) \longrightarrow \Z^{N+1}/(I - A_1)\Z^{N+1}
\quad
\text{and}
\quad
\quad\eta^{\mathcal{O}}_0: \K_0(\OAT) \longrightarrow \Z^{N}/(I - A)\Z^{N}
\end{equation*}
by
$ 
\eta_0^\T([t_i t_i^*]) = [e_i], \,\,
\eta_0^\T([p_0]) = [e_0]
$ 
and 
$
\eta_0^{\mathcal{O}}([s_i s_i^*) = [f_i],
$
where
$e_i =(0,\dots, 0, \overset{i+1}{1}, 0,\dots,0) \in \Z^{N+1}, i=0,1,\dots, N,$
and
$f_i =(0,\dots, 0, \overset{i}{1}, 0,\dots,0) \in \Z^N, i=1,\dots, N.$
Then
both
 $\eta^\T_0: \K_0(\T_{A^t}) \longrightarrow \Z^{N+1}/(I - A_1)\Z^{N+1}$ 
 and
 $\eta^{\mathcal{O}}_0: \K_0(\OAT) \longrightarrow \Z^{N}/(I - A)\Z^{N}$ 
 are isomorphisms such that the diagram 
\begin{equation*}
\begin{CD}
  \K_0(\calK(H)) @ = \Z \\
 @V{\iota_*}VV   @VV{\iota_{A_1}}V \\
 {\K}_0(\TAT) @>{{\eta}^{\mathcal{T}}_0}>>   \Z^{N+1}/ (I-A_1)\Z^{N+1}  \\
 @V{\pi_{A^t*}}VV   @VV{q_{A_1}}V \\
  {\K}_0(\OAT) @>{\eta^{\mathcal{O}}_0}>>\Z^N/ (I-A)\Z^N  %\\
 %@VVV   @VVV \\
 % 0 @. 0
 \end{CD}
 \end{equation*}
 commutes, and $\eta_0^{\T}([1_{\TAT}]) = [1_{N+1}]$ holds, where
 $[1_{\TAT}]$ denotes the class in $\K_0(\TAT)$ of the unit $1_{\TAT}$ 
 of $\TAT$ and $[1_{N+1}] = [(1,\dots,1)]$ 
 denotes the class in $\Z^{N+1}/(I - A_1)\Z^{N+1}$ of the vector $(1,\dots,1)\in \Z^{N+1}$.
\end{lemma}
\begin{proof}
As
$\eta^\T_0: \K_0(\T_{A^t}) \longrightarrow \Z^{N+1}/(I - A_1)\Z^{N+1}
$
and
$
\eta^{\mathcal{O}}_0: \K_0(\OAT) \longrightarrow \Z^{N}/(I - A)\Z^{N}
$
are both isomorphisms,
it suffices to show that the commutativity of the diagram.
%The surjecive homomorphism
%$\pi_{A^t*}: \K_0(\TAT) \longrightarrow \K_0(\OAT)$
%is defined by
%$$
%\pi_{A^t*}([p_0])= 0, \qquad
%\pi_{A^t*}([t_i t_i^*]) = [s_i s_i^*] \quad \text{ for } i=1,\dots,N.
%$$
%As the surjective homomorphism
%$q_{A_1}: \Z^{N+1}/(I - A_1)\Z^{N+1} \longrightarrow
%\Z^{N}/(I - A)\Z^{N}$is defined by
%$q_{A_1}(z_0,z_1,\dots,z_N) = (z_1,\dots,z_N)$,
The commutativity
$q_{A_1}\circ \eta^{\mathcal{T}}_0 = \eta^{\mathcal{O}}_0\circ \pi_{A^t*}$
is immediate by definition of $\pi_{A^t}$.
In the short exact sequence \eqref{eq:transposeToeplitaext},
% $\calK(H)$ is a $C^*$-subalgebra of $\TAT$
% such that 
%$t_1 p_0 t_1^*$ is a 
%minimal projection whose 
the class $[t_1 p_0 t_1^*]$ in $\K_0(\calK(H))$
corresponds to the positive generator of $\Z.$
As 
$[t_1 p_0 t_1^*] 
= [(t_1p_0) (t_1p_0)^*]
=[(t_1p_0)^* (t_1p_0)] =[p_0]$
in $\K_0(\TAT)$, 
 we have
 $$
 \eta^{\mathcal{T}}_0(\iota_*([t_1 p_0 t_1^*]) )
 = \eta^{\mathcal{T}}_0([ p_0 ])
 =[e_0].
 $$
 On the other hand,
 as 
 $$
 \hat{\iota}_A(1) = (I-A) 
 \begin{bmatrix}
 1\\
 0\\
 \vdots\\
 0
 \end{bmatrix}
 =
 \begin{bmatrix}
 1 -A(1,1)\\
 -A(2,1)\\
 \vdots \\
 -A(N,1)
 \end{bmatrix}
\quad
\text{ and hence }
\quad
\iota_{A_1}(1) 
 =
 \begin{bmatrix}
 0 \\
 1 -A(1,1)\\
 -A(2,1)\\
 \vdots \\
 -A(N,1)
 \end{bmatrix}
$$
so that 
$$
\iota_{A_1}(1) -[e_0]
 =
 \begin{bmatrix}
 -1 \\
 1 -A(1,1)\\
 -A(2,1)\\
 \vdots \\
 -A(N,1)
 \end{bmatrix}
= 
(I - A_1) 
 \begin{bmatrix}
 0 \\
 1\\
 0\\
 \vdots \\
 0
 \end{bmatrix}.
$$
This shows that 
$
\iota_{A_1}(1) =[e_0]
$ in 
$
\Z^{N+1}/(I - A_1)\Z^{N+1}
$  
and
$
 \eta^{\mathcal{T}}_0(\iota_*([t_1 p_0 t_1^*])) 
 = \iota_{A_1}(1),
 $
proving that the diagram
\begin{equation*}
\begin{CD}
  \K_0(\calK(H)) @ = \Z \\
 @V{\iota_*}VV   @VV{\iota_{A_1}}V \\
 {\K}_0(\TAT) @>{{\eta}^{\mathcal{T}}_0}>>   \Z^{N+1}/ (I-A_1)\Z^{N+1}  
  \end{CD}
 \end{equation*}
commutes.
The equality
$
  \eta_0^{\T}([1_{\TAT}]) 
%= \eta_0^\T(\sum_{i=1}^N [t_i t_i^*] + [p_0]) 
%= \sum_{i=0}^N [e_i] 
= [1_{N+1}]
$
is obvious.
\end{proof}
Let $\partial: \K_1(\OAT) \longrightarrow \K_0(\calK(H))$
be the connecting map of the K-theory cyclic six term exact sequence for the extension 
\eqref{eq:transposeToeplitaext}.
\begin{lemma}\label{lem:eta4.5}
There are isomorphisms
\begin{equation*}
\eta^\T_1: \K_1(\TAT) \longrightarrow \Ker(s_A: \Ker(I-A) \longrightarrow \Z)
\quad
\text{and}
\quad
\eta^{\mathcal{O}}_1: \K_1(\OAT) \longrightarrow \Ker(I-A)
\end{equation*}
such that the diagram 
\begin{equation}\label{eq:K1TAT}
\begin{CD}
 %0 @. 0 \\
 %@VVV   @VVV \\
 \K_1(\TAT) @>{\eta^{\mathcal{T}}_1}>> \Ker(s_A: \Ker(I-A)\longrightarrow \Z) \\
 @V{\pi_{A^t*}}VV   @VV{\iota_{s_A}}V \\
 \K_1(\OAT) @>{\eta^{\mathcal{O}}_1}>>  \Ker(I-A) \\
 @V{\partial}VV   @VV{s_A}V \\
 \K_0(\calK(H))  @ = \Z \\
\end{CD}
\end{equation}
commutes.
\end{lemma}
\begin{proof}
Let
$\zeta_1^\mathcal{O}: \K_1(\OAT)\ni [U(x)] \longrightarrow x=(x_i)_{i=1}^N \in \Ker(I-A)$
be the isomorphism defined by the inverse of
the isomorphism 
from $\Ker(I-A)$ to $\K_1(\OAT)$
 explicitly defined in \cite[P. 33]{Ro}.
%$ x \in \Ker(I-A) \longrightarrow [U(x)] \in \K_1(\OAT)$.
As in Lemma \ref{lem:theorem3.3},
the class $[\TAT]_w$ in $\Ext_w(\OAT)( = \K^1(\OAT))$ corresponds to the class
$[(-1,\dots,-1)]$ of
$(-1,\dots,-1)$ in $\Z^N/(I-A^t)\Z^N.$
Then the connecting map
$\partial: \K_1(\OAT) \longrightarrow \K_0(\calK(H)) = \Z$
is given by the pairing 
\begin{equation}\label{eq:delU}
\partial([U]) = < [-1,\dots,-1] \, \, | \, \, \zeta^{\mathcal{O}}_1([U])>, \qquad [U] \in \K_1(\OAT)
\end{equation}
where the right  hand side above is the inner product.
For $[U(x)] \in \K_1(\OAT)$, we have
\begin{equation*}
\partial([U(x)]) = < [-1,\dots,-1] \, \, | \, \,  x> = -\sum_{j=1}^N x_j
\quad \text{ for } x \in \Ker(I-A).
\end{equation*} 
By setting
$
\eta^{\mathcal{O}}_1([U]) = \zeta_1^\mathcal{O}([U^*])
$
for
$ 
[U] \in \K_1(\OAT),
$
we have
\begin{equation*}
\partial([U(x)]) 
%= < [-1,\dots,-1], \zeta^{\mathcal{O}}_1(U(x))>
= < [1,\dots,1] \, \, | \, \,  \eta^{\mathcal{O}}_1([U(x)])>
=\sum_{j=1}^N x_j
=(s_A\circ \eta^{\mathcal{O}}_1)([U(x)]),
\end{equation*}
proving
$\partial = s_A \circ \eta^{\mathcal{O}}_1.$

Since both the homomorphisms
$\pi_{A^t*}:\K_1(\T_{A^t}) \longrightarrow  \K_1(\OAT)$
and
$\iota_{s_A}: \Ker(s_A: \Ker(I-A)\longrightarrow \Z)\longrightarrow  \Ker(I-A)$
are injective,
there exists a unique isomorphism
$
\eta^\T_1: \K_1(\TAT) \longrightarrow \Ker(s_A: \Ker(I-A) \longrightarrow \Z)
$ satisfying
$\iota_{s_A} \circ \eta^\T_1 = \eta^{\mathcal{O}}_1\circ \pi_{A^t*},$
showing the commutative diagram 
\eqref{eq:K1TAT}.
\end{proof}
We thus have
\begin{proposition}\label{prop:Toeplitzexact}
The $\K$-theory cyclic six term exact sequence
\begin{equation}\label{eq:sixtermTAT}
\begin{CD}
0 @>>> \K_1(\T_{A^t}) @>{\pi_{A^t*}}>> \K_1(\mathcal{O}_{A^t}) \\
@AAA  @.   @VV{\partial}V  \\
 \K_0(\mathcal{O}_{A^t})  @<{\pi_{A^t*}}<< \K_0(\T_{A^t}) @<{\iota_*}<< \K_0(\calK(H)) 
\end{CD}
\end{equation}
for the Toeplitz extension \eqref{eq:transposeToeplitaext}
is computed to be the cyclic six term exact sequence \eqref{eq:sixtermA1},
such that there is a commutative diagram
\begin{equation}\label{eq:Toeplitzexact}
\begin{CD}
 0 @. 0 \\
 @VVV   @VVV \\
 \K_1(\TAT) @>{\eta^{\mathcal{T}}_1}>> \Ker(s_A: \Ker(I-A)\longrightarrow \Z) \\
 @V{\pi_{A^t*}}VV   @VV{\iota_{s_A}}V \\
 \K_1(\OAT) @>{\eta^{\mathcal{O}}_1}>>  \Ker(I-A) \\
 @V{\partial}VV   @VV{s_A}V \\
 \K_0(\calK(H))  @= \Z \\
 @V{\iota_*}VV   @VV{\iota_{A_1}}V \\
  \K_0(\TAT) @>{\eta^{\mathcal{T}}_0}>>   \Z^{N+1}/(I - A_1)\Z^{N+1}  \\
 @V{\pi_{A^t*}}VV   @VV{q_{A_1}}V \\
  \K_0(\OAT) @>{\eta^{\mathcal{O}}_0}>>\Z^N/ (I-A)\Z^N  \\
 @VVV   @VVV \\
  0 @. 0
 \end{CD}
 \end{equation}
where
 the horizontal arrows
 $\eta^{\mathcal{T}}_1, \eta^{\mathcal{O}}_1, \eta^{\mathcal{T}}_0, \eta^{\mathcal{O}}_0
 $
 are all isomorphisms 
 satisfying 
 $\eta^{\mathcal{T}}_0([1_{\TAT}]) = [1_{N+1}],$
 and $[1_\TAT]$ denotes the class in $\K_0(\TAT)$ of the unit $1_{\TAT}$ of $\TAT.$  
\end{proposition}

%%%%%%%%%%%%%%%%%%%%%%%%%%%%%%%%%%%%%%%%%%%%%%%%%%%%%%%%%%%%%%%%%%%%%%
%{\scriptsize
%\begin{equation*}
%\begin{CD}
% 0 @>>>\K_1(\T_{A^t}) @>>> \K_1({\mathcal{O}}_{A^t}) @>>> \K_0(\calK(H))  @>>> \K_0(\T_{A^t}) @>>> \K_0({\mathcal{O}}_{A^t}) @>>> 0  \\
% @. @VV{\widetilde{\Phi}^0}V  @VV{\Phi^0}V @|   @VV{\widetilde{\Phi}^1}V  @VV{\Phi^1}V  @.  \\
%0 @>>>\Ker(I-\widehat{A})/{i_1(\Z)} @>>> \Ker(I-A) @>>> \Z @>>> \Z^N/ (I-\widehat{A})\Z^N  @>>> \Z^N/ (I-A)\Z^N  @>>> 0  \\
%\end{CD}
%\end{equation*}
%}
%%%%%%%%%%%%%%%%%%%%%%%%%%%%%%%%%%%%%%%%%%%%%%%%%

%%%%%%%%%%%%%%%%%%%%%%%%%%%%%%%%%%%%%%%%%%%%%%%%%%%%%%%%
%%%%%%%%%%%%%%%%%%%%%%%%%%%%%%%%%%%%%%%%%%
\section{K-theoretic duality for Toeplitz extensions}
%%%%%%%%%%%%%%%%%%%%%%%%%%%%%%%%%%%%%%
We have four kinds of cyclic six term exact sequences
\eqref{eq:6termA}, \eqref{eq:sixtermOA}, \eqref{eq:sixtermA1} and \eqref{eq:sixtermTAT}.
The first two are isomorphic by 
\eqref{eq:Khomocyclic},
and the second two are isomorphic by \eqref{eq:Toeplitzexact}.
In this section, we will show that \eqref{eq:sixtermOA} and \eqref{eq:sixtermA1} 
are isomorphic to prove that  all of the four cyclic six term exact sequences are isomorphic.
\begin{lemma}
There exists an isomorphism
$\xi^0: \Z^N/ (I - \widehat{A})\Z^N \longrightarrow \Z^{N+1}/(I-A_1)\Z^{N+1}
$
satisfying  
 $\xi^0(\hat{\iota}_A(1) + [1_N]) = [1_{N+1}]$
 such that 
 the diagram 
\begin{equation*}
\begin{CD}
 \Z  @= \Z \\
 @V{\hat{\iota}_A}VV   @VV{\iota_{A_1}}V \\
  \Z^N/ (I - \widehat{A})\Z^N @>{\xi^0}>>   \Z^{N+1}/(I-A_1)\Z^{N+1}  \\
 @V{q_{\widehat{A}}}VV   @VV{q_{A_1}}V \\
  \Z^N/( I - A)\Z^N @ = \Z^N/ (I-A)\Z^N  \\
 \end{CD}
 \end{equation*}
commutes.
%where
%$q_{A_1}: \Z^{N+1}/(I-A_1)\Z^{N+1}\longrightarrow \Z^N/( I - A)\Z^N$
%is defined by $q_{A_1}(z_0,z_1,\dots,z_N) = (z_1,\dots,z_N) .$
\end{lemma}
\begin{proof}
Let $U, V$ be the $(N+1)\times (N+1)$ matrices defined by 
\begin{equation*}%\label{eq:UV}
U
=
\begin{bmatrix}
1        & 0      &\cdots &\cdots & 0  \\
1- A(1,1)& 1      &\ddots &       &\vdots \\
-A(2,1)  & 0      &\ddots &\ddots &\vdots \\
\vdots   & \vdots &\ddots &\ddots & 0  \\
-A(N,1)  & 0      &\cdots &   0   & 1
\end{bmatrix},
\qquad
V
=
\begin{bmatrix}
0     & 1    & 0     & \cdots &\cdots &  0  \\
-1    & 0    & -1    & \cdots &\cdots & -1 \\
0     & 0    & 1     & 0      &\cdots & 0 \\
\vdots&      &\ddots &\ddots  & \ddots &\vdots \\
\vdots&      &       &\ddots  &\ddots  & 0 \\
0     &\cdots&\cdots &\cdots  &0       & 1
\end{bmatrix}.
\end{equation*} 
%so that
%\begin{equation*}
%U^{-1}
%=
%\begin{bmatrix}
%1        & 0     &\cdots &\cdots &0  \\
%A(1,1)-1 & 1     & \ddots&       &\vdots \\
%A(2,1)   & 0     &\ddots &\ddots &\vdots \\
%\vdots   &\vdots &\ddots &\ddots &0 \\
%A(N,1)   & 0     &\cdots &   0   & 1
%\end{bmatrix}.
%\end{equation*}
For $(x_1,\dots,x_N) \in \Z^N$, define
$\xi^0(x_1,\dots,x_N)= (0, x_1,\dots,x_N)\in \Z^{N+1}$.
Note that 
$U^{-1}(0, x_1,\dots,x_N)= (0, x_1,\dots,x_N)\in \Z^{N+1}$
holds.
As
$
U( I -A_1) V = 
\left[
\begin{smallmatrix}
1     & 0  & \cdots   & 0 \\
0     &    &          & \\
\vdots&    & I - \widehat{A}  &  \\
0          &          &        
\end{smallmatrix}
\right],
$
we have 
$$
( I - A_1) V [0\oplus \Z^N] = U^{-1}[0\oplus (I - \widehat{A})\Z^N] = \xi^0((I - \widehat{A})\Z^N)
$$
so that 
$
\xi^0((I- \widehat{A})\Z^N) \subset ( I - A_1) \Z^{N+1}.
$
Hence 
$\xi^0: \Z^N \longrightarrow \Z^{N+1}$ 
induces a homomorphism 
$\Z^N/ (I - \widehat{A})\Z^N \longrightarrow \Z^{N+1}/(I-A_1)\Z^{N+1}
$ that is still denoted by $\xi^0$.
It gives rise to an isomorphism
because
$$
\Z^{N+1}/(I-A_1)\Z^{N+1} 
= U \Z^{N+1}/ U (I-A_1)\Z^{N+1} 
= \Z^{N+1}/( \Z \oplus (I-\widehat{A})\Z^N)
=\Z^N/ (I - \widehat{A})\Z^N .
$$
Since
$$
\hat{\iota}_A(1) + [1_N] 
= (I- A) 
\begin{bmatrix}
1\\
0\\
\vdots\\
0
\end{bmatrix}
+
\begin{bmatrix}
1\\
1\\
\vdots\\
1
\end{bmatrix}
=
\begin{bmatrix}
1-A(1,1) +1\\
-A(2,1) +1\\
\vdots\\
-A(N,1) +1
\end{bmatrix},
$$
%and $U^{-1}(0, x_1,\dots,x_N) = (0, x_1,\dots,x_N)$,
we have
$$
\xi^0(\hat{\iota}_A(1) + [1_N]) - [1_{N+1}]
=
\begin{bmatrix}
-1\\
1-A(1,1) \\
-A(2,1) \\
\vdots\\
-A(N,1) 
\end{bmatrix}
= (I - A_1) 
\begin{bmatrix}
0\\
1 \\
0 \\
\vdots\\
0 
\end{bmatrix}
$$
so that we have
$\xi^0(\hat{\iota}_A(1) + [1_N]) = [1_{N+1}]
$ in $\Z^{N+1} / ( I - A_1) \Z^{N+1}.$
The equalities 
$q_{\widehat{A}} = q_{A_1}\circ \xi^0$ and
$\iota_{A_1} = \xi^0 \circ \iota_{\widehat{A}}$
are obvious.
\end{proof}
Recall that a homomorphism
$j_A: \Ker(I - \widehat{A})/\iota_1(\Z) \longrightarrow \Ker(I-A)$
is defined in \eqref{eq:ijs},
 and $\iota_{s_A}:\Ker(s_A: \Ker(I-A)\longrightarrow \Z) \longrightarrow
 \Ker(I-A)$
stands for the inclusion map.
% defined by
% $\iota_{s_A}([x_i]_{i=1}^N) =[x_i]_{i=1}^N).$
Define 
$\tilde{\xi}^0: 
\Ker(I - \widehat{A})/ \iota_1(\Z) \longrightarrow \Ker( s_A:\Ker(I -A) \longrightarrow \Z)
$
by $j_A$.
Since 
the diagram
\begin{equation}\label{eq:tildexi0}
\begin{CD}
 \Ker(I - \widehat{A})/\iota_1(\Z) @>{\tilde{\xi}^0}>> \Ker(s_A: \Ker(I-A)\longrightarrow \Z) \\
 @V{j_A}VV   @VV{\iota_{s_A}}V \\
 \Ker(I-A) @ =  \Ker(I-A) \\
 \end{CD}
 \end{equation}
commutes,
we have the following proposition.
%By the exact sequence
%\eqref{eq:6termA}, we know that 
%$ j_A(\Ker( I - \widehat{A})) = \Ker(s_A:\Ker(I -A) \longrightarrow \Z)$
%so that $\tilde{\xi}^0$ is well-defined and 
%the diagram \eqref{eq:tildexi0} is commutative,
%because $\iota_{s_A} = \id.$ 
\begin{proposition}\label{prop:sixtermexact}
We have  a commutative diagram
\begin{equation}\label{eq:sixtermexact}
\begin{CD}
 0 @. 0 \\
 @VVV   @VVV \\
 \Ker(I - \widehat{A})/\iota_1(\Z) @>{\tilde{\xi}^0}>> \Ker(s_A: \Ker(I-A)\longrightarrow \Z) \\
 @V{j_A}VV   @VV{\iota_{s_A}}V \\
 \Ker(I-A) @ =  \Ker(I-A) \\
 @V{s_A}VV   @VV{s_A}V \\
 \Z  @= \Z \\
 @V{\hat{\iota}_A}VV   @VV{\iota_{A_1}}V \\
  \Z^N/ (I - \widehat{A})\Z^N @>{\xi^0}>>   \Z^{N+1}/(I-A_1)\Z^{N+1}  \\
 @V{q_{\widehat{A}}}VV   @VV{q_{A_1}}V \\
  \Z^N/( I - A)\Z^N @ = \Z^N/ (I-A)\Z^N  \\
 @VVV   @VVV \\
  0 @. 0
 \end{CD}
 \end{equation}
 such that 
 $\xi^0(\hat{\iota}_A(1) + [1_N]) = [1_{N+1}].$
\end{proposition}
Let us define the following four isomorphisms through 
\eqref{eq:Khomocyclic},
\eqref{eq:Toeplitzexact}
and \eqref{eq:sixtermexact}
by
\begin{equation*}%\label{eq:sixtermexact}
{\small
\begin{CD}
\widetilde{\Phi}^0_A: \widetilde{\K}^0(\OA) @>{\widetilde{\phi}^0}>> 
\Ker(I - \widehat{A})/\iota_1(\Z) @>{\tilde{\xi}^0}>> \Ker(s_A: \Ker(I-A)\longrightarrow \Z) 
 @>{\eta_1^{\T-1}}>> \K_1(\TAT), \\
 {\Phi}_A^0: \K^0(\OA) @>{\phi^0}>> \Ker(I-A) @ =  \Ker(I-A) @>{\eta_1^{\calO-1}}>>  \K_1(\OAT),\\
 \widetilde{\Phi}_A^1:  \widetilde{\K}^1(\OA) @>{\widetilde{\phi}^1}>> 
 \Z^N/ (I - \widehat{A})\Z^N @>{\xi^0}>>   \Z^{N+1}/(I-A_1)\Z^{N+1} @>{\eta_0^{\T-1}}>> \K_0(\TAT),   \\
 \Phi^1_A: {\K}^1(\OA) @>{\phi^1}>> \Z^N/( I - A)\Z^N @ = \Z^N/ (I-A)\Z^N @>{\eta_0^{\calO-1}}>> \K_0(\OAT). 
  \end{CD}}
 \end{equation*}
Therefore we reach the following theorem.
\begin{theorem}\label{thm:OATAT}
We have  a commutative diagram
\begin{equation}
\begin{CD}
 0 @. 0 \\
 @VVV   @VVV \\
 \widetilde{\K}^0(\OA) @>{\widetilde{\Phi}_A^0}>> \K_1(\TAT) \\
 @V{j_A^0}VV   @VV{\pi_{A^t*}}V \\
 \K^0(\OA) @>{\Phi_A^0}>>  \K_1(\OAT) \\
 @V{\iota_{\mathbb{C}}^*}VV   @VV{\partial}V \\
 \K^0(\mathbb{C}) @ = \K_0(\calK(H)) \\
 @V{\iota_A}VV   @VV{\iota_*}V \\
  \widetilde{\K}^1(\OA) @>{\widetilde{\Phi}_A^1}>> \K_0(\TAT)  \\
 @V{j_A^1}VV   @VV{\pi_{A^t*}}V \\
  {\K}^1(\OA) @>{\Phi_A^1}>> \K_0(\OAT) \\
 @VVV   @VVV \\
  0 @. 0
 \end{CD}
\end{equation}
where two vertical sequences are exact and 
all of the horizontal arrows are isomorphisms such that 
\begin{equation*}
\widetilde{\Phi}_A^1([\TA]_s) = - [1_{\TAT}] \qquad \text{ in} \quad  \K_0(\TAT).
\end{equation*}
\end{theorem}
\begin{proof}
By combining the three commutatuve diagrams
\eqref{eq:Khomocyclic},
\eqref{eq:Toeplitzexact}
 and \eqref{eq:sixtermexact},
 we get the disired commutative diagram 
 from  Proposition \ref{prop:Khomocyclic},
Proposition \ref{prop:Toeplitzexact}
and Proposition \ref{prop:sixtermexact}.
The equalities
$
\widetilde{\Phi}_A^1([\TA]_s) 
= (\eta_0^{\T -1}\circ \xi^0\circ \widetilde{\phi}^1)([\TA]_s)
= - [1_{\TAT}]
$
in
$\K_0(\TAT)$
are obvious.
\end{proof}
%%%%%%%%%%%%%%%%%%%%%%%%%%%%%%%%%%%%%%%%%%%%
%%%%%%%%%%%%%%%%%%%%%%%%%%%%%%%%%%%%%%%%%%
\section{K-theoretic duality for extensions}
%%%%%%%%%%%%%%%%%%%%%%%%%%%%%%%%%%%%%%
We now reach the following definition.
\begin{definition}\label{def:Kduality}
Let $\A$ and $\A^\prime$ be separable unital nuclear $C^*$-algebras.
Two extensions
\begin{equation*}
\tau:
 \quad 0 \longrightarrow 
\calK(H) \overset{\iota}{\longrightarrow}
 \E \overset{q}{\longrightarrow} 
 \A \longrightarrow 0
\quad \text{ and }
\quad
\tau^\prime:  
\quad 0 \longrightarrow 
\calK(H) \overset{\iota'}{\longrightarrow}
 \E^\prime \overset{q'}{\longrightarrow}
  \A^\prime \longrightarrow 0 
\end{equation*}
are said to be a $\K$-{\it theoretic duality pair}\/ 
if there are horizontal arrows of isomorphisms in the diagrams below  making them commutative
\begin{equation*} 
\begin{CD}
 0 @. 0 \\
 @VVV   @VVV \\
 \widetilde{\K}^0(\A) @>{\widetilde{\Phi}^0}>> \K_1(\E^\prime) \\
 @V{j_\A^0}VV   @VV{q_*}V \\
 \K^0(\A) @>{\Phi^0}>>  \K_1(\A^\prime) \\
 @V{\iota_{\mathbb{C}}^*}VV   @VV{\partial}V \\
 \K^0(\mathbb{C}) @>{\Phi^0_{\mathbb{C}}}>> \K_0(\calK(H)) \\
 @V{\iota_\A}VV   @VV{\iota_*}V \\
  \widetilde{\K}^1(\A) @>{\widetilde{\Phi}^1}>> \K_0(\E^\prime)  \\
 @V{j_{\A}^1}VV   @VV{q_*}V \\
  {\K}^1(\A) @>{\Phi^1}>> \K_0(\A^\prime) \\
 @VVV   @VVV \\
  0 @. 0
 \end{CD}
 \qquad
\text{ and }
\qquad
\begin{CD}
 0 @. 0 \\
 @VVV   @VVV \\
 \widetilde{\K}^0(\A^\prime) @>{\widetilde{\Phi}^{\prime0}}>> \K_1(\E) \\
 @V{j_{\A'}^0}VV   @VV{q_*}V \\
 \K^0(\A^\prime) @>{\widetilde{\Phi}^{\prime1}}>>  \K_1(\A) \\
 @V{\iota_{\mathbb{C}}^*}VV   @VV{\partial}V \\
 \K^0(\mathbb{C}) @>{\Phi^{\prime0}_{\mathbb{C}}}>> \K_0(\calK(H)) \\
 @V{\iota_{\A'}}VV   @VV{\iota_*}V \\
  \widetilde{\K}^1(\A^\prime) @>{\widetilde{\Phi}^{\prime1}}>> \K_0(\E)  \\
 @V{j_{\A'}^1}VV   @VV{q_*}V \\
  {\K}^1(\A') @>{\Phi^{\prime1}}>> \K_0(\A)  \\
 @VVV   @VVV \\
  0 @. 0.
 \end{CD}
 \end{equation*}
%such that all horizontal arrows are isomorphisms.
If in particular $\Phi^{0}_{\mathbb{C}} = \Phi^{\prime0}_{\mathbb{C}} =\id$
under the identification $\K^0(\mathbb{C}) = \K_0(\calK(H)) = \Z$,
and 
\begin{equation}\label{eq:def.Phitaus1}
\widetilde{\Phi}^1([\tau]_s)  = - [1_{\E^\prime}] \text{ in } \K_0(\E^\prime), \qquad
\widetilde{\Phi}^{\prime 1}([\tau^\prime]_s)  = - [1_{\E}] \text{ in } \K_0(\E),
\end{equation}
then they are said to be {\it strong}\/ $\K$-{\it theoretic duality pair}\/,
 where $ [1_{\E'}], [1_\E]$ denote the classes in $\K_0(\E'), \K_0(\E)$ of the units of 
 $\E', \E$, respectively.
\end{definition}
For the strong K-theoretic duality pair,
the above commutative diagrams say that  
the equalities \eqref{eq:def.Phitaus1}
imply
\begin{equation}\label{eq:def.Phitauw1}
{\Phi}^1([\tau]_w)  = - [1_{\A^\prime}] \text{ in } \K_0(\A^\prime), \qquad
{\Phi}^{\prime 1}([\tau^\prime]_w)  = - [1_{\A}] \text{ in } \K_0(\A).
\end{equation}
By virtue of the introduction of the notion of 
$\K$-theoretic duality for extensions of unital nuclear $C^*$-algebras,  
Theorem \ref{thm:OATAT} implies Theorem \ref{thm:main2}.
%\begin{theorem}\label{thm:main2}
%Let $A$ be an irreducible non permutation matrices.
%The Toeplitz extensions $\tau_A$ and $\tau_{A^t}$ of the Cuntz--Krieger algebras 
%$\OA$ and $\mathcal{O}_{A^t}$ are strong K-theoretic duality pair. 
%\end{theorem}
We note the following proposition.
\begin{proposition}\label{prop:Kdualityclass}
Let $\A$ and $\A^\prime$ be separable unital nuclear $C^*$-algebras.
Let
\begin{equation*}
\tau_i:
 \,\,\, 0 \longrightarrow 
\calK(H) {\longrightarrow}
 \E_i {\longrightarrow} 
 \A \longrightarrow 0
\quad \text{ and }
\quad
\tau_i^\prime:  
\,\,\, 0 \longrightarrow 
\calK(H) {\longrightarrow}
 \E_i^\prime {\longrightarrow}
  \A^\prime \longrightarrow 0 
\end{equation*}
be extensions for $i=1,2$.
 Suppose that  $[\tau_1]_s = [\tau_2]_s $ in $\Es(\A)$
 and
  $[\tau'_1]_s = [\tau'_2]_s $ in $\Es(\A^\prime)$.
 If $\tau_1$ and $\tau_1'$ are K-theoretic duality pair
 (resp. strong K-theoretic duality pair),
 then
  $\tau_2$ and $\tau_2'$ are K-theoretic duality pair
 (resp. strong K-theoretic duality pair).
 Hence the (strong) K-theoretic duality pair depends only 
 on the strong equivalence classes of extensions. 
\end{proposition}
\begin{proof}
Assume that $[\tau_1]_s = [\tau_2]_s$ in $\Es(\A)$
and $[\tau'_1]_s = [\tau'_2]_s $ in $\Es(\A)$.
We may find unitaries $U, U' \in \B(H)$ and
 isomorphisms $\gamma: \E_1 \longrightarrow \E_2$,
$\gamma': \E'_1 \longrightarrow \E'_2$ such that the diagrams
\begin{align*}\label{eq:exttwo}
{\begin{CD}
0 @>>> \calK(H)  @>{\iota_1}>> \E_1 @>{q_1}>> \A @>>> 0\\ 
@. @V{\Ad(U)}VV  @V{\gamma}VV  \parallel @. @. \\
0 @>>> \calK(H)  @>{\iota_2}>> \E_2 @>{q_2}>> \A @>>> 0  
\end{CD}} \\
\intertext{and}
{\begin{CD}
0 @>>> \calK(H)  @>{\iota'_1}>> \E'_1 @>{q'_1}>> \A @>>> 0\\ 
@. @V{\Ad(U')}VV  @V{\gamma'}VV  \parallel @. @. \\
0 @>>> \calK(H)  @>{\iota'_2}>> \E'_2 @>{q'_2}>> \A @>>> 0  
\end{CD}}
\end{align*}
commute.
%%%%%%%%%%%%%%%%%%%%%%%%%%%%
%Hence we have commutative diagrams
%\begin{equation*} 
%\begin{CD}
% 0 @. 0 \\
% @VVV   @VVV \\
% \K_1(\E_1) @>{\gamma_*}>> \K_1(\E_2) \\
% @V{q_{1*}}VV   @VV{q_{2*}}V \\
% \K_1(\A) @ =   \K_1(\A) \\
% @V{\partial}VV   @VV{\partial}V \\
% \K_0(\calK(H)) @>{\Ad(U)_*}>> \K_0(\calK(H)) \\
% @V{\iota_{1*}}VV   @VV{\iota_{2*}}V \\
% \K_0(\E_1) @>{\gamma_*}>> \K_0(\E_2)  \\
% @V{q_{1*}}VV   @VV{q_{2*}}V \\
% \K_0(\A) @ = \K_0(\A)  \\
% @VVV   @VVV \\
%  0 @. 0
% \end{CD}
% \qquad
%\text{ and }
%\qquad
%\begin{CD}
% 0 @. 0 \\
% @VVV   @VVV \\
% \K_1(\E'_1) @>{\gamma'_*}>> \K_1(\E'_2) \\
% @V{q'_{1*}}VV   @VV{q'_{2*}}V \\
% \K_1(\A') @ =   \K_1(\A') \\
% @V{\partial}VV   @VV{\partial}V \\
% \K_0(\calK(H)) @>{\Ad(U')_*}>> \K_0(\calK(H)) \\
% @V{\iota'_{1*}}VV   @VV{\iota'_{2*}}V \\
% \K_0(\E'_1) @>{\gamma'_*}>> \K_0(\E'_2)  \\
% @V{q'_{1*}}VV   @VV{q'_{2*}}V \\
% \K_0(\A') @ = \K_0(\A')  \\
% @VVV   @VVV \\
%  0 @. 0.
% \end{CD}
% \end{equation*}
% By using the above commutative diagrams, 
% we know that  the hypothesis 
%%%%%%%%%%%%%
Hence it is a direct consequence 
 that $\tau_1$ and $\tau_1'$ are K-theoretic duality pair
 implies that $\tau_2$ and $\tau_2'$ are K-theoretic duality pair.
  
Suppose further that $\tau_1$ and $\tau'_1$ are strong K-theoretic pair.
Let $\widetilde{\Phi}_1^1:\widetilde{\K}^1(\A_1) \longrightarrow \K_0(\E'_1)$
and 
  $\widetilde{\Phi}_1^{\prime 1}:\widetilde{\K}^1(\A'_1) \longrightarrow \K_0(\E_1)$
be isomorphisms giving rise to the strong K-theoretic duality pair  for $\tau_1$ and $\tau'_1$
such that 
$
\widetilde{\Phi}_1^1([\tau_1]_s) = -[1_{\E'_1}] \text{ in } \K_0(\E'_1)
$
 and
 $
\widetilde{\Phi}_1^{\prime 1}([\tau'_1]_s) = -[1_{\E_1}] \text{ in } \K_0(\E_1).
$
As $[\tau_1]_s = [\tau_2]_s$ and
 $[\tau'_1]_s = [\tau'_2]_s$,
 we have
$
 (\gamma'_*\circ\widetilde{\Phi}_1^1)([\tau_2]_s) = -[\gamma'_*(1_{\E'_1})] = -[1_{\E'_2}],
$ 
and symmetrically
$
(\gamma_*\circ\widetilde{\Phi}_1^{\prime 1})([\tau'_2]_s)= -[1_{\E_2}].
$
Since $\Ad(U)_* =\Ad(U')_* =\id$ on $\K_0(\calK(H))$,
%we conclude that 
the isomorphisms
$\widetilde{\Phi}_2^1:= \gamma'_*\circ\widetilde{\Phi}_1^1 
:\widetilde{\K}^1(\A_2) \longrightarrow \K_0(\E'_2)$
and 
  $\widetilde{\Phi}_2^{\prime 1}:=\gamma_*\circ\widetilde{\Phi}_1^{\prime 1 }
  :\widetilde{\K}
  ^1(\A'_2) \longrightarrow \K_0(\E_2)$
yield 
 that $\tau_2$ and $\tau'_2$ are strong K-theoretic pair.
 \end{proof}

%%%%%%%%%%%%%%%%%%%%%%%%%%%%%%%%%%%%%%%%%%%%%%%%%%%%%%%%%%%%%
%{\scriptsize
%\begin{equation*}
%\begin{CD}
%0 @>>>\widetilde{\K}^0(\A)@>>>{\K}^0(\A) @>>> \K^0(\mathbb{C}) @>>> \widetilde{\K}^1(\A) @>>>{\K}^1(\A) @>>> 0 \\
%@. @VV{\widetilde{\Phi}^0}V  @VV{\Phi^0}V @|   @VV{\widetilde{\Phi}^1}V  @VV{\Phi^1}V  @.  \\
%0 @>>>\K_1( \E^\prime)@>>> \K_1(\A^\prime) @>>> \K_0(\calK(H)) @>>> \K_0(\E^\prime) @>>> \K_0(\A^\prime) @>>> 0  
%\end{CD}
%\end{equation*}
%}
%and
%{\scriptsize
%\begin{equation*}
%\begin{CD}
%0 @>>>\widetilde{\K}^0(\A^\prime)@>>>{\K}^0(\A^\prime) @>>> \K^0(\mathbb{C}) @>>> \widetilde{\K}^1(\A^\prime) @>>>{\K}^1(\A^\prime) @>>> 0 \\
%@. @VV{\widetilde{\Phi}^0}V  @VV{\Phi^0}V @|  @VV{\widetilde{\Phi}^1}V  @VV{\Phi^1}V  @.  \\
%0 @>>>\K_1( \E)@>>> \K_1(\A) @>>> \K_0(\calK(H)) @>>> \K_0(\E) @>>> \K_0(\A) @>>> 0  
%\end{CD}
%\end{equation*}
%}
%%%%%%%%%%%%%%%%%%%%%%%%%%%%%%%%%%%%%%%%%%%%%
%%%%%%%%%%%%%%%%%%%%%%%%%%%%%%%%%%%%%%%%%%%%%%%%%%%%%%%%%%%%%%%%%%%%%%
%%%%%%%%%%%%%%%%%%%%%%%%%%%%%%%%%%%%%%%%%%%%%
\section{Duality for Toeplitz extensions}
%%%%%%%%%%%%%%%%%%%%%%%%%%%%%%%%%%%%%%%%%%%%%%%
We will first show the following proposition, that is a sort of duality 
of the matrix $A$ and its transpose $A^t$ from the viewpoint of Cuntz--Krieger algebras
and extension groups.
\begin{proposition}
Let $A, B$ be irreducible non permutation matrices with entries in $\{0,1\}.$
The following are equivalent.
\begin{enumerate}
\renewcommand{\theenumi}{(\roman{enumi})}
\renewcommand{\labelenumi}{\textup{\theenumi}}
\item 
The Cuntz--Krieger algebras $\OA$ and $\OB$ are isomorphic, and 
there exists an isomorphism $\varphi: \Ew(\OA) \longrightarrow \Ew(\OB)$
such that 
$\varphi([\TA]_w) = [\TB]_w.$ 
\item 
The Cuntz--Krieger algebras $\OAT$ and $\OBT$ are isomorphic, and 
there exists an isomorphism $\phi: \Ew(\OAT) \longrightarrow \Ew(\OBT)$
such that 
$\phi([\TAT]_w) = [\TBT]_w.$ 
\end{enumerate}
\end{proposition}
\begin{proof}
Under the identification 
between $\K_0(\OA)$ and $\Ew(\OAT)$, 
the class $[1_{\OA}]$ of the unit $1_{\OA}$ of $\OA$ correspondes to 
the class $-[\TAT]_w$ of the minus of the Toeplitz extensioon of $\OAT.$
Hence the equivalence between (i) and (ii) is immediate from R{\o}rdam's classification
result of the Cuntz--Krieger algebras (\cite{Ro}).   
\end{proof}
The following lemma is elementary, so we omit its proof.
\begin{lemma}\label{lem:No4}
Let $\tau_i: \A_i \longrightarrow \Q(H)$ 
be the Busby invariants for extensions
$$
0 \longrightarrow \calK(H) 
\longrightarrow \E_i 
\longrightarrow \A_i 
\longrightarrow 0, \qquad i=1,2.
$$
Suppose that 
there exist isomorphisms
$\alpha: \calK(H) \longrightarrow \calK(H),$
$\beta: \E_1 \longrightarrow \E_2,$
$\gamma: \A_1 \longrightarrow \A_2$
of $C^*$-algebras such that 
the diagarm
\begin{equation*}%\label{eq:No4}
\begin{CD}
0 @>>>  \calK(H) @>>> \E_1 @>>> \A_1 @>>> 0 \\
@. @V{\alpha}VV  @V{\beta}VV  @VV{\gamma}V  @.  \\
0 @>>>  \calK(H) @>>> \E_2 @>>> \A_2 @>>> 0 
\end{CD}
\end{equation*}
commutes.
Then $\gamma: \A_1 \longrightarrow \A_2$ induces isomorphisms
$\gamma_s^*:\Es(\A_2) \longrightarrow \Es(\A_1)$
and
$\gamma_w^*:\Ew(\A_2) \longrightarrow \Ew(\A_1)$
such that 
$\gamma_s^*([\tau_2]_s) = [\tau_1]_s$,
$\gamma_w^*([\tau_2]_w) = [\tau_1]_w$
and the diagram
\begin{equation*}
\begin{CD}
\Z @>{\iota_{\A_2}}>> \Es(\A_2) @>{q_2}>> \Ew(\A_2) \\
 @|  @V{\gamma_s^*}VV  @VV{\gamma_w^*}V    \\
 \Z @>{\iota_{\A_1}}>> \Es(\A_1) @>{q_1}>> \Ew(\A_1)
\end{CD}
\end{equation*}
commutes.
\end{lemma}
We thus have the following proposition.
\begin{proposition}\label{prop:No7}
The following are equivalent.
\begin{enumerate}
\renewcommand{\theenumi}{(\roman{enumi})}
\renewcommand{\labelenumi}{\textup{\theenumi}}
\item
 There exists an isomorphism 
$\Phi:\OA\longrightarrow \OB$ 
of $C^*$-algebras such that the induced isomorphism
$\Phi_s^*: \Es(\OB) \longrightarrow \Es(\OA)$
satsifies 
$\Phi_s^*([\TA]_s) = [\TB]_s.$
\item There exist isomorphisms 
$\Phi:\OA\longrightarrow \OB$ and $\Psi:\TA\longrightarrow \TB$
of $C^*$-algebras such that the diagram 
\begin{equation}\label{eq:CDTATBOAOB7}
\begin{CD}
\TA @>{\pi_A}>> \OA \\
@V{\Psi}VV @VV{\Phi}V \\ 
\TB @>{\pi_B}>> \OB 
\end{CD}
\end{equation}
commutes.
\end{enumerate}
\end{proposition}
\begin{proof}
Let us denote by $\tau_A: \OA\longrightarrow \Q(H)$
and
$\tau_B: \OB\longrightarrow \Q(H)$
the Busby invariants for the Toeplitz extensions $\TA$ and $\TB$,
respectively.

(i) $\Longrightarrow$ (ii):
Since $\Phi_s^*([\tau_B]_s) = [\tau_A]_s$,
there exists a unitary
$U \in \B(H)$ such that 
$\tau_B (\Phi(X)) = \pi(U)\tau_A(X)\pi(U)^*, X \in \OA$.
As
\begin{gather*}
\T_A = \{( T, X) \in \B(H) \oplus \OA \mid \pi(T) = \tau_A(X) \}, \\
\T_B = \{( T', X') \in \B(H) \oplus \OB \mid \pi(T') = \tau_B(X') \},
\end{gather*}
we may define the map
$\Psi: \TA\longrightarrow \TB$ by setting
$\Psi(T,X) = (U T U^*,\Phi(X)).$
It is easy to see that 
$\Psi(T,X) \in \TB$ for $(T,X) \in \TA$
and 
$\Psi: \TA\longrightarrow \TB$ 
gives rise to an isomorphism such that 
the diagram
\eqref{eq:CDTATBOAOB7} commutes.

(ii) $\Longrightarrow$ (i):
This implication follows from 
Lemma \ref{lem:No4}.
\end{proof}
Now we give a proof of Theorem \ref{thm:main3}.

\noindent
{\it Proof of Theorem \ref{thm:main3}.}
(i) $\Longrightarrow$ (ii):
Assume that there exist 
isomorphisms 
$\Phi:\OA\longrightarrow \OB$ and $\Psi:\TA\longrightarrow \TB$
of $C^*$-algebras such that the diagram \eqref{eq:CDTATBOAOB} commutes.
Since $\TA$ has exactly one non-trivial ideal, 
which is the kernel of the 
surjection $\pi_A:\TA\longrightarrow \OA$.
It is isomorphic to the $C^*$-algebra $\calK(H_A)$ 
of compact operators on 
the  sub Fock space $H_A$. 
%so we write
%$ 
%0 \longrightarrow 
%\calK(H_A) \longrightarrow 
%\TA \longrightarrow 
%\OA \longrightarrow 0.$
Hence the isomorphism
$\Psi:\TA\longrightarrow \TB$
satisfies $\Psi(\calK(H_A)) = \calK(H_B).$
Put
$\alpha = \Psi|_{\calK(H_A)}: \calK(H_A) \longrightarrow \calK(H_B)$
the restriction of $\Psi$ to $\calK(H_A)$.
%Since $\alpha$ maps a minimal projection  
%of $\calK(H_A)$ to a minimal projection of $\calK(H_B)$,
It induces an isomorphism from $\K_0(\calK(H_A))$ 
to $\K_0(\calK(H_B))$,
that maps $\Z$ to $\Z$ 
as the identity.
Therefore we have a commutative diagram:
\begin{equation}\label{eq:CDNo6}
\begin{CD}
0 @>>>  \calK(H_A) @>>> \TA @> \pi_A>> \OA @>>> 0 \\
@. @V{\alpha}VV  @V{\Psi}VV  @VV{\Phi}V  @.  \\
0 @>>>  \calK(H_B) @>>> \TB @> \pi_B>> \OB @>>> 0 
\end{CD}
\end{equation}
such that 
$\alpha_* = \id: \K_0(\calK(H_A))=\Z \longrightarrow  \K_0(\calK(H_B))=\Z.$
In the general exact sequence \eqref{eq:ZSW},
%\begin{equation*}
%\begin{CD}
%\Z  @>{\iota_\A}>>  \Es(\A) @>{q_\A}>>\Ew(\A)
%\end{CD}
%\end{equation*}
%at the middle of a unital nuclear $C^*$-algebra $\A$ is given by 
the homomorphism
$\iota_\A:\Z\longrightarrow \Es(\A)$ is given
by
$\iota_\A: m \in \Z \longrightarrow [\sigma_m] \in \Es(\A)
$
where $\sigma_m = \Ad(u_m) \circ \tau :\A \longrightarrow \Q(H)$
for a trivial extension $\tau: \A \longrightarrow \Q(H)$ and a unitary 
$u_m \in \Q(H)$ of Fredholm index $m$.
We denote by
$\iota_A, \iota_B$ the homomorphisms $\iota_{\OA}, \iota_{\OB}$ in
\eqref{eq:ZSW} for $\A = \OA$,
respectively.
For 
the isomorphism
$\Phi:\OA\longrightarrow \OB$,
define 
isomorphisms
$\Phi_s^*:\Es((\OB)\longrightarrow\Es(\OA)$
and
$\Phi_w^*: \Ew(\OB) \longrightarrow \Ew(\OA)$
by setting 
$\Phi_s^*([\sigma]_s) = [\sigma\circ \Phi]_s$
and
$\Phi_w^*([\sigma]_w) = [\sigma\circ \Phi]_w$,
respectively,
for a unital $*$-monomorphism
$\sigma:\OB\longrightarrow \Q(H)$.
As the equality
$
\iota_A(m) = \Phi_s^*\circ \iota_B(m),  m \in \Z
$
holds,
we have the commutative diagram:
\begin{equation*}
\begin{CD}
\Z @>{\iota_A}>> \Es(\OA) @>{q_A}>> \Ew(\OA) \\
 @|  @A{\Phi_s^*}AA  @AA{\Phi_w^*}A    \\
 \Z @>{\iota_B}>> \Es(\OB) @>{q_B}>> \Ew(\OB).
\end{CD}
\end{equation*}
By virtue of Proposition \ref{prop:No7}, the commutative diagram
\eqref{eq:CDNo6}
tells us that
 $\Phi_s^*([\T_B]_s) = [\TA]_s$ in $\Es(\OA)$ and
$\Phi_w^*([\T_B]_w) = [\TA]_w$ in $\Ew(\OA)$
so that 
$\Phi_s^*(\iota_B(1) + [\TB]_s) = \iota_A(1) + [\TA]_s.$
Under the identifications
$\widetilde{\K}^1(\, \cdot\, ) = \Es(\, \cdot\, )$
and
${\K}^1(\, \cdot\, ) = \Ew(\, \cdot\, )$,
the fact that 
the Toeplitz extensions are strong K-theoretic duality pair
gives  the following commutative diagram
\begin{equation}
\begin{CD}
 0 @. 0 @. 0 @. 0 \\
 @VVV   @VVV @VVV   @VVV \\
 \K_1(\TBT) @<{\widetilde{\Phi}_B^0}<<   \widetilde{\K}^0(\OB) 
@>>>  \widetilde{\K}^0(\OA) @>{\widetilde{\Phi}_A^0}>> \K_1(\TAT) \\
 @V{\pi_{B^t*}}VV   @VV{j_B^0}V @V{j_A^0}VV   @VV{\pi_{A^t*}}V \\
 \K_1(\OBT) @<{\Phi_B^0}<<  \K^0(\OB) @>>>  \K^0(\OA) @>{\Phi_A^0}>>  \K_1(\OAT) \\
 @V{\partial}VV     @VV{\iota_{\mathbb{C}}^*}V @V{\iota_{\mathbb{C}}^*}VV   @VV{\partial}V \\
 \K_0(\calK(H_{B^t}))  @ = \K_0(\mathbb{C}) @= \K^0(\mathbb{C}) @ = \K_0(\calK(H_{A^t})) \\
 @V{\iota_*}VV   @VV{\iota_B}V @V{\iota_A}VV   @VV{\iota_*}V \\
 \K_0(\TBT) @<{\widetilde{\Phi}_B^1}<<  \widetilde{\K}^1(\OB) @>{\Phi_s^*}>> 
\widetilde{\K}^1(\OA) @>{\widetilde{\Phi}_A^1}>> \K_0(\TAT)  \\
 @V{\pi_{B^t*}}VV   @VV{j_B^1}V @V{j_A^1}VV   @VV{\pi_{A^t*}}V \\
 \K_0(\OBT)  @<{\Phi_B^1}<< \K^1(\OB) @>{\Phi_w^*}>> {\K}^1(\OA) @>{\Phi_A^1}>> \K_0(\OAT) \\
 @VVV   @VVV @VVV   @VVV \\
  0 @. 0  @. 0 @. 0.
 \end{CD}
\end{equation}
By putting
\begin{gather*}
\Phi^{\mathcal{T}}_0 := \widetilde{\Phi}_A^1 \circ \Phi_s^* \circ \widetilde{\Phi}_B^{1-1} 
 : \K_0(\TBT) \longrightarrow \K_0(\TAT), \\
\Phi^{\mathcal{O}}_0 := {\Phi}_A^1 \circ \Phi_w^* \circ \Phi_B^{1 -1}
 : \K_0(\OBT) \longrightarrow \K_0(\OAT),
\end{gather*}
 we have a commtative diagram
\begin{equation}
\begin{CD}
 0 @. 0 \\
 @VVV   @VVV \\
\K_1(\TBT) @>>> \K_1(\TAT) \\
 @V{\pi_{B^t*}}VV   @VV{\pi_{A^t*}}V \\
 \K_1(\OBT) @>>>  \K_1(\OAT) \\
 @V{\partial}VV   @VV{\partial}V \\
 \K_0(\calK(H_{B^t})) @ = \K_0(\calK(H_{A^t})) \\
 @V{\iota_*}VV   @VV{\iota_*}V \\
  \K_0(\TBT) @>{\Phi_0^{\mathcal{T}}}>> \K_0(\TAT)  \\
 @V{\pi_{B^t*}}VV   @VV{\pi_{A^t*}}V \\
  {\K}_0(\OBT) @>{\Phi_0^{\mathcal{O}}}>> \K_0(\OAT) \\
 @VVV   @VVV \\
  0 @. 0
 \end{CD}
\end{equation}
such that 
$\Phi_0^{\mathcal{T}}([1_{\TBT}]) = [1_{\TAT}]$
and
$\Phi_0^{\mathcal{O}}([1_{\OBT}]) = [1_{\OAT}]$.
The the groups 
$\K_0(\TAT)$ and $\K_0(\TBT)$
consist of the equivalence classes of projections 
in $\TAT$ and $\TBT$, respectively,
Thanks to the classification result  \cite[Theorem 5.3]{ERR2013} (cf. \cite{ERRS})
proved by Eilers--Restroff--Ruiz,
we may find an isomorphism
$\varphi:\TBT \longrightarrow \TAT$
such that 
$\varphi_* = \Phi_{0}^{\mathcal{T}} :\K_0(\TBT) \longrightarrow \K_0(\TAT)$.
Since the isomorphism
$\varphi:\TBT \longrightarrow \TAT$
preserves their ideals of compact operators,
we see that 
$\varphi(\calK(H_{B^t})) = \calK(H_{A^t})$,
so that 
the isomorphism
$\Psi^t := \varphi^{-1}: \TAT\longrightarrow \TBT$
induces an isomorphism
$\Phi^t:\OAT\longrightarrow \OBT$
such that the diagram \eqref{eq:CDTATBOAOB2}
commutes. \qed
%Since an isomorphism from $\TA$ to $\TB$ automatically
%induces the commutative diagram in Theorem \ref{thm:main3} (i),
%we may show the following corollary.
%\begin{corollary}\label{cor:TATB}
%Let $A, B$ be irreducible non permutation matrices with entries in $\{0,1\}.$
%Then the Toeplitz algebras $\TA$ and $\TB$ are isomorphic if and only if 
%the Toeplitz algebras $\TAT$ and $\TBT$ defined by their transposed matrices are isomorphic.
%\end{corollary}

%\begin{remark}
%Looking back to the proofs of Theorem \ref{thm:main3}
%with Corollary \ref{cor:main4}, we notice that 
%one may actually prove the following:
% Suppose that  extensions 
 %$0 \longrightarrow \calK(H) \longrightarrow \E_i\longrightarrow \A_i \longrightarrow 0$ and 
%$0 \longrightarrow \calK(H) \longrightarrow  \E_i^\prime\longrightarrow \A_i^\prime \longrightarrow 0$
%are strong K-theoretic duality pair 
% for each $i=1,2$, and the $C^*$-algebras $\A_i$ and $\A_i^\prime$ are Kirchberg algebras
% satisfying UCT. If $\E_1$ and $\E_2$ are isomorphic, then 
%   $\E_1^\prime$ and $\E_2^\prime$ are isomorphic. Because 
%the Eilers--Restroff--Ruiz's classification result \cite[Corollary 4.16]{ERR2013} works for 
%Kirchberg algebras satisfying UCT in this setting, one may follow the proof
%of Theorem \ref{thm:main3} in a similar way. 
%\end{remark}
 We finally present several equivalence conditions under which two Toeplitz algebras are isomorphic.
\begin{proposition} \label{prop:classification TA}
Let $A=[A(i,j)]_{i,j=1}^N, B=[B(i,j)]_{i,j=1}^M$ be irreducible non permutation matrices with entries in $\{0,1\}.$
Then the following four assertions are isomorphic.
\begin{enumerate}
 \renewcommand{\theenumi}{(\roman{enumi})}
\renewcommand{\labelenumi}{\textup{\theenumi}}
\item
$\TA$ is isomorphic to $\TB$.
\item
There exist isomorphisms 
$\Phi:\OA\longrightarrow \OB$ and $\Psi:\TA\longrightarrow \TB$
of $C^*$-algebras such that the diagram 
\begin{equation}\label{eq:CDTATBOAOB2}
\begin{CD}
\TA @>{\pi_A}>> \OA \\
@V{\Psi}VV @VV{\Phi}V \\ 
\TB @>{\pi_B}>> \OB 
\end{CD}
\end{equation}
commutes.
\item
There exists an isomorphism
$\varphi: \Z^{N+1}/(I- A_1)\Z^{N+1} \longrightarrow \Z^{M+1}/(I- B_1)\Z^{M+1}$
of groups such that 
$\varphi([(1,\dots,1)]) =[(1,\dots,1)]$
and
$\varphi([(1,0,\dots,0)]) =[(1,0,\dots,0)].$
\item
There exists an isomorphism 
$\Phi:\OA\longrightarrow \OB$ of $C^*$-algebras 
such that 
$\Phi^*_s([\TB]_s) = [\TA]_s$ in $\Es(\OA)$
\end{enumerate}
where $A_1$ is the $(N+1)\times (N+1)$ matrix
and $B_1$ is the $(M+1) \times (M+1)$ matrix  
defined by
$
A_1
=\left[
\begin{smallmatrix}
1     & 1  & \cdots   & 1 \\
0     &    &          & \\
\vdots&    & A        &  \\
0          &          &        
\end{smallmatrix}
\right]
$
and
$
B_1
=\left[
\begin{smallmatrix}
1     & 1  & \cdots   & 1 \\
0     &    &          & \\
\vdots&    & B        &  \\
0          &          &        
\end{smallmatrix}
\right],
$
respectively.
\end{proposition} 
\begin{proof}
(i) $\Longleftrightarrow$ (ii) is clear.

(ii) $\Longleftrightarrow$ (iv) is seen in Proposition \ref{prop:No7}.

(ii) $\Longrightarrow$ (iii):
 Assume the condition (ii), 
so that we have a commutative diagram \eqref{eq:CDNo6}.
Put 
$$
\varphi: 
= \Psi_*: \K_0(\TA)=\Z^{N+1}/(I- A_1)\Z^{N+1} \longrightarrow 
\K_0(\OB) =\Z^{M+1}/(I- B_1)\Z^{M+1}.
$$ 
Since the position $[1_{\TA}]$ of the unit of $\TA$ in $\K_0(\TA)$
is the class $[(1,\dots,1)]$ of the vector $(1,\dots,1)$ in 
the group $\Z^{N+1}/(I- A_1)\Z^{N+1}$,
we have $\varphi([(1,\dots,1)]) =[(1,\dots,1)]$.
Since 
$\Psi(\TA) = \TB$ and 
 the class of a minimal projection
is the class $[(1,0,\dots,0)]$ of the vector
$(1,0,\dots,0)$
in $\K_0(\TA)=\Z^{N+1}/(I- A_1)\Z^{N+1}$,
and similarly in  
$\K_0(\TB) =\Z^{M+1}/(I- B_1)\Z^{M+1}$,
we have 
$\varphi([(1,0,\dots,0)]) =[(1,0,\dots,0)].$

(iii) $\Longrightarrow$ (i):
Assume (iii).
Since the class $[(1,0,\dots,0)]$ 
of the vector $(1,0,\dots,0)$ is $[e_0]$ in the notation of 
Lemma \ref{lem:etaTO},
the condition 
$\varphi([(1,0,\dots,0)]) =[(1,0,\dots,0)]$
implies that 
$
\varphi: Z^{N+1}/(I- A_1)\Z^{N+1} \longrightarrow \Z^{M+1}/(I- B_1)\Z^{M+1}
$
induces an isomorphism 
$
\varphi_0:
 Z^{N}/(I- A)\Z^{N} \longrightarrow \Z^{M}/(I- B)\Z^{M}
$
which sends $[(1,\dots,1)]$ in 
$
 Z^{N}/(I- A)\Z^{N}
$ 
to
 $[(1,\dots,1)]$ in 
$
 Z^{M}/(I- B)\Z^{M}.
$ 
Hence there exists an isomorphism $\Phi:\OAT\longrightarrow \OBT$
such that $\Phi_* = \varphi_0:\K_0(\OAT)\longrightarrow \K_0(\OBT)$
by \cite{Ro}. 
It also induces an isomorphism
 $\Phi_*: \K_1(\OAT) =\Ker(I - A) \longrightarrow\K_1(\OBT)=\Ker(I-B)$
 written $\varphi_1$.
Let $\zeta_1^\mathcal{O}: \K_0(\OAT) \longrightarrow \Ker(I-A)$
be the isomorphism defined in the proof of Lemma \ref{lem:eta4.5}.
Let $\partial_{A^t}: \K_1(\OAT)\longrightarrow \K_0(\calK(H))$
denote the connectiong map $\partial$ in the diagram \eqref{eq:K1TAT}.
For $[U]\in \K_1(\OBT)$, by \eqref{eq:delU} we have
\begin{align*}
\partial_{A^t}(\Phi_*([U])) 
=& < [-1,\dots,-1] \, \, | \, \, \zeta^{\mathcal{O}}_1([\Phi(U)])> 
= < [-1,\dots,-1] \, \, | \, \, \Phi_*(\zeta^{\mathcal{O}}_1([U])> \\
=&  < \varphi([-1,\dots,-1]) \, \, | \, \, \zeta^{\mathcal{O}}_1([U])> 
=  < [-1,\dots,-1]) \, \, | \, \, \zeta^{\mathcal{O}}_1([U])> 
= \partial_{B^t}([U])
\end{align*}
because of the condition
$\varphi([-1,\dots,-1]) = [-1,\dots,-1].$
The map
$\iota_{A_1}:\Z\longrightarrow Z^{N+1}/(I- A_1)\Z^{N+1}$
is given by by $m \in \Z\longrightarrow m[e_0] \in Z^{N+1}/(I- A_1)\Z^{N+1}$.
Hence we have a commutative diagram
\begin{equation*}
\begin{CD}
\Ker(I-B) @<{\eta_1^\mathcal{O}}<< \K_1(\OBT) @<{\Phi_*}<< \K_1(\OAT)  @>{\eta_1^\mathcal{O}}>> \Ker(I-A) \\
@V{s_B}VV @V{\partial_{B^t}}VV @VV{\partial_{A^t}}V @VV{s_A}V \\ 
\Z @ = \K_0(\calK(H)) @ =  \K_0(\calK(H)) @ = \Z.
\end{CD}
\end{equation*}
 We thus have a commutative diagram 
\begin{equation}\label{eq:twoToeplitzexact}
\begin{CD}
 0 @. 0 \\
 @VVV   @VVV \\
\Ker(s_B: \Ker(I-B)\longrightarrow \Z) @<<< \Ker(s_A: \Ker(I-A)\longrightarrow \Z) \\
 @V{\iota_{s_B}}VV   @VV{\iota_{s_A}}V \\
 \Ker(I-B)  @<{\varphi_1}<<  \Ker(I-A) \\
 @V{s_B}VV   @VV{s_A}V \\
\Z  @= \Z \\
 @V{\iota_{B_1}}VV   @VV{\iota_{A_1}}V \\
\Z^{M+1}/(I - B_1)\Z^{M+1} @<{\varphi}<<   \Z^{N+1}/(I - A_1)\Z^{N+1}  \\
 @V{q_{B_1}}VV   @VV{q_{A_1}}V \\
 \Z^M/ (I-B)\Z^M @<{\varphi_0}<< \Z^N/ (I-A)\Z^N  \\
 @VVV   @VVV \\
  0 @. 0
 \end{CD}
 \end{equation}
 where the horizontal arrows in \eqref{eq:twoToeplitzexact} are all isomorphisms. 
 The the groups 
$\K_0(\TAT)$ and $\K_0(\TBT)$
consist of the equivalence classes of projections 
in $\TAT$ and $\TBT$, respectively,
Since $\varphi([1,\dots,1]) = [1,\dots,1]$ in $\K_0(\TBT)$
and $\varphi_0([1,\dots,1]) = [1,\dots,1]$ in $\K_0(\OBT)$,
%together with 
%Lemma \ref{lem:positiveK}, 
by  the classification result  \cite[Theorem 5.3]{ERR2013} (cf. \cite{ERRS})
proved by Eilers--Restroff--Ruiz,
we may find an isomorphism
$\varphi:\TBT \longrightarrow \TAT$,
because of Proposition \ref{prop:Toeplitzexact}.
Therefore we conclude that $\TA$ is isomorphic to $\TB$ by Corollary \ref{cor:main4}.
\end{proof}

%%%%%%%%%%%%%%%%%%%%%%%%%%%%%%%%%%%%%%%
%%%%%%%%%%%%%%%%%%%%%%%%%%%%%%%%%%%%%%%%%%%%%
\section{Examples} \label{sect:Examples}
%%%%%%%%%%%%%%%%%%%%%%%%%%%%%%%%%%%%%%%%%%%%%%%
We present examples of the triplet
$(\Z^{N+1}/ ( I - A_1)\Z^{N+1}, [(1,0,\dots,0)], [(1,\dots,1)])$ 
that shows $(\K_0(\TAT), \iota_*([1]), [1_{\TAT}])$
for several matrices, 
%that show the triple
%$(\K_0(\TAT), \iota_*([1]), [1_{\TAT}]),$ 
% in the following,  
computed without difficulty by hand.

{\bf 1.}
\begin{itemize}
\item 
 $B =
\begin{bmatrix}
1 & 1 \\
1 &1
\end{bmatrix},
\quad
(\Z^{3}/ ( I - B_1)\Z^{3}, [(1,0,0)], [(1,1,1)])
\cong(\Z, 1, 0).$
\item 
 $B_{-} =
\begin{bmatrix}
1 & 1 &0 & 0\\
1 & 1 &1 & 0 \\
0 & 1 & 1 &1 \\
0 & 0 & 1 &1 
\end{bmatrix},
\quad
(\Z^{5}/ ( I - B_{-1})\Z^{5}, [(1,0,0,0,0)], [(1,1,1,1,1)])
\cong(\Z, 1, -1).$
\item 
 $F =
\begin{bmatrix}
1 & 1 \\
1 & 0  
\end{bmatrix},
\quad
(\Z^{3}/ ( I - F_1)\Z^{3}, [(1,0,0)], [(1,1,1)])
\cong(\Z, 1, -2).$
\item 
 $C =
\begin{bmatrix}
1 & 1 &1\\
1 & 1 &1\\
1 & 0 &1  
\end{bmatrix},
\quad
\quad
(\Z^{4}/ ( I - C_1)\Z^{4}, [(1,0,0,0)], [(1,1,1,1)])
\cong(\Z, 1, -1).$
\item 
 $D =
\begin{bmatrix}
0 & 1 &1\\
1 & 1 &1\\
1 & 0 &1  
\end{bmatrix},
\quad
\quad
(\Z^{4}/ ( I - D_1)\Z^{4}, [(1,0,0,0)], [(1,1,1,1)])
\cong(\Z, 1, 0).$
 \end{itemize}
 Hence we have
 $
 \calT_B \cong \calT_{D^t}
 $ and 
 $ 
 \calT_{B_{-}}\cong \calT_{C^t},
 $
 whereas 
$\calT_{F}$ is not isomorphic to any of 
$
 \calT_B,  \calT_{B_{-}},  \calT_{C^t}, \calT_{D^t}.
 $
 Since $\K_0(\OAT) \cong \K_0(\TAT)/\Z\iota_*([1]),$
  we have
 $\calO_B \cong \calO_{B_{-}}\cong \calO_F\cong\calO_{C^t} \cong \calO_{D^t},
 $
 that are all isomorphic to $\calO_2$.

{\bf 2.}
$$
A =
\begin{bmatrix}
1 & 1 & 1\\
1 & 1 &1\\
1 & 0 &0  
\end{bmatrix},
\qquad
B = A^t
=
\begin{bmatrix}
1 & 1 & 1\\
1 & 1 &0\\
1 & 1 &0  
\end{bmatrix},
$$
\begin{gather*}
(\Z^{4}/ ( I - A_1)\Z^{4}, [(1,0,0,0)], [(1,1,1,1)])
\cong(\Z, 2, -2), \\
(\Z^{4}/ ( I - B_1)\Z^{4}, [(1,0,0,0)], [(1,1,1,1)])
\cong(\Z\oplus \Z/2\Z, 1 \oplus 0, (-1) \oplus 1 ).
\end{gather*}
Hence we have 
$\calT_A$ is not isomorphic to $\calT_B$,
 and
$\OA$ is not isomorphic to $\OB$. 
We indded see that 
$\OA$ is isomorphic to $\calO_3$, 
and
$\OB$ is isomorphic to $\calO_3\otimes M_2(\mathbb{C})$
(cf. \cite{EFW1981}, \cite{PS}).

\medskip

{\it Acknowledgment:}
%The author would like to thank Joachim Cuntz for his useful comments and suggestions
%on a preliminary version of this paper.
This work was supported by JSPS KAKENHI 
Grant Number 19K03537.

\end{document}